\documentclass[11pt]{article}
\usepackage[utf8]{inputenc}

\usepackage{amssymb,amsfonts}
 \usepackage[normalem]{ulem} 
\usepackage{enumitem} %for flexibility in enumerate
\usepackage{amsmath} %for DeclareMathOperator
\usepackage[hyphens]{url} \usepackage[colorlinks=true,linkcolor=blue,citecolor=blue,urlcolor=blue]{hyperref} 
\usepackage{amsthm}
\usepackage{fullpage}
\usepackage{graphicx}
\usepackage{xcolor}
\usepackage{comment}
\usepackage{subcaption}
\usepackage{todonotes}
\usepackage{color}
\usepackage{authblk}
\usepackage{epstopdf}

\usepackage[backend=biber,style=numeric-comp]{biblatex}
\providecommand{\keywords}[1]{\textbf{\textit{Keywords---}} #1}
\providecommand{\MSC}[1]{\textbf{\textit{MSC---}} #1}

\title{Spectral Properties of Infinitely Smooth Kernel Matrices in the Single Cluster Limit, with Applications to Multivariate Super-Resolution}
\author[1]{Nuha Diab}
\author[1]{Dmitry Batenkov}
\affil[1]{Department of Applied Mathematics, Tel Aviv University, Tel Aviv, Israel}

\date{\today}

\newcommand{\ci}{{\cal X}} % a node configuration
\newcommand{\di}{{\cal Y}} % recaled node configuration
\newcommand{\sss}{{\cal S}} % a sampling set 
\newcommand{\Gg}{{\cal G}} % grid sampling set
\newcommand{\td}[1][d]{\mathbb{T}^{#1}} % d-dim. torus

\newcommand{\ZZ}{\mathbb{Z}}
\newcommand{\NN}{\mathbb{N}}
\newcommand{\PP}{\mathbb{P}}
\newcommand{\HH}{\mathbb{H}}
\newcommand{\RR}{\mathbb{R}}
\newcommand{\TT}{\mathbb{T}}
\newcommand{\CC}{\mathbb{C}}

\newcommand{\srf}{\text{SRF}}

\DeclareMathOperator{\sep}{sep} % separation of a set
\DeclareMathOperator{\rank}{rank} % rank
\DeclareMathOperator{\diag}{diag} % diagonal matrix
\DeclareMathOperator{\sinc}{sinc}
\DeclareMathOperator{\Sinc}{Sinc}

\newtheorem{theorem}{Theorem}[section]     % reset each section
\newtheorem{lemma}[theorem]{Lemma}
\newtheorem{proposition}[theorem]{Proposition}
\newtheorem{corollary}[theorem]{Corollary}
\theoremstyle{definition}
\newtheorem{definition}[theorem]{Definition}

\theoremstyle{remark}
\newtheorem{remark}[theorem]{Remark}
\begin{document}

\maketitle

\begin{abstract}
    We study the spectral properties of infinitely smooth multivariate kernel matrices when the nodes form a single cluster. We show that the geometry of the nodes plays an important role in the scaling of the eigenvalues of these kernel matrices. For the multivariate Dirichlet kernel matrix, we establish a criterion for the sampling set ensuring precise scaling of eigenvalues. Additionally, we identify specific sampling sets that satisfy this criterion. Finally, we discuss the implications of these results for the problem of super-resolution, i.e. stable recovery of sparse measures from bandlimited Fourier measurements.
\end{abstract}

\par \MSC{15A18, 65T40, 65F20, 47A55, 47A75, 47B34}

\section{Introduction}

    \textbf{Super-Resolution Problem.} The problem of super-resolution (SR) is to recover the fine details of a signal from inherently low resolution measurements in the frequency domain.  The SR problem has a variety of applications in signal processing, imaging, optics, inverse scattering, and data analysis problems \cite{Bertero1996,lindberg2012}. In the last years, a particular SR model has received considerable attention, where the signal is modelled as a sparse measure \cite{batenkov2021a}. Let $\TT^d := [-\pi,\pi)^d\equiv \left(\RR\mod 2\pi\right)^d$, let $\delta(\cdot)$ denote the Dirac delta distribution {\color{black}(in the sense of generalized functions \cite{strichartz2003})}, and consider the signal
    \begin{equation} \label{eq:spike}
        f(t) := \sum_{j=1}^n\alpha_j\delta (t-x_j), \quad t, x_j \in \TT^d, \quad \alpha_j \in \CC {\color{black}.}
    \end{equation}
    We refer to $\{x_j\}$ as the "nodes" and to $\{\alpha_j\}$ as the "coefficients" or ``amplitudes''. The noisy measurements are given by
    \begin{equation} \label{eq:samples}
        \hat{f}(\omega) := \sum_{j=1}^n \alpha_j e^{\imath \langle \omega,x_j \rangle}+\epsilon(\omega), \quad \omega \in \Gg_N := \{-N+1,\dots,-1,0,1,\dots,N-1\}^d \subset \ZZ^d,\quad |\epsilon(\omega)|\leq \epsilon.
    \end{equation}
    {\color{black}The problem of super-resolution is, given discrete noisy Fourier measurements \eqref{eq:samples},
    recover the parameters of model \eqref{eq:spike}, i.e. $\{\alpha_j,x_j\}$.}
    \par Moreover, define $\srf := \frac{1}{N\Delta}$, where $\Delta$ represents the smallest distance (e.g. in the infinity norm) between the nodes. It is well-known that, at least in the one-dimensional case, this quantity controls the numerical stability of the problem. The super-resolution regime is when $\srf \gg 1$; while the well-separated case is when $\srf \leq 1$. In the super-resolution regime, nodes are organized in clusters, with a distance of order $\Delta$ between any two nodes within a cluster.

    \par One primary objective in theory of super-resolution is to define optimal bounds for reconstruction errors, commonly referred to as min-max error bounds, or the computational resolution limits \cite{donoho1992a,liu2021,Li&Liao2020, demanet2015,batenkov2021a}. These bounds are achieved by the {\color{black} (intractable) ``oracle''} algorithm under the worst-case scenario. Furthermore, these stability bounds serve to validate the optimality of different tractable solution methods, offering assurances regarding the methods' performance \cite{li2020a, katz2023b, katz2024b}. It has been shown previously that the min-max bounds are related to the smallest singular value of the Vandermonde matrix $U$ of the system \cite{demanet2015}, which is defined as
    \begin{equation} \label{eq:Van}
        U(x_1,\dots,x_n;\Gg_N) := \big[\exp{\imath \langle \omega, x_j \rangle}\big]_{\omega \in \Gg_N}^{j=1,\dots,n}\in\CC^{|\Gg_N|\times n}{\color{black}, \  \langle \omega, x_j \rangle := \omega^* x_j = \sum_{k=1}^d \omega_k^* x_{j,k},}
    \end{equation}
    {\color{black}where $\omega^*$ denotes the conjugate transpose of $\omega$.}

    \par Extensive research has been conducted on the one-dimensional {\color{black}$(d=1)$} scenario, covering both well-separated and super-resolution regimes \cite{batenkov2021, Li&Liao2020,cuyt2018,petz2021}. 
    Let $\ell$ be the number of nodes in the largest cluster. Then, in the super-resolution regime, the smallest singular value of the Vandermonde matrix scales like $\srf^{1- \ell}$ {\color{black}whenever $N\Delta \geq c(n)$}, resulting in ``on-grid'' min-max bounds to be of order $\srf^{2\ell -1}\epsilon$ \cite{batenkov2018, Li&Liao2020}, (cf. similar off-grid bounds \cite{batenkov2021a,liu2022}).
     In the multidimensional context, several results are available for the well-separated regime as well as the super-resolution regime \cite{liu2021,kunis2021,kunis2020a,garcia2020,poon2019}, under specific conditions related to the unknown nodes. A key distinction between the one-dimensional and multidimensional cases lies in the spectrum of the Vandermonde matrix, wherein the geometry of the nodes plays a crucial role. As we shall show in this paper, and consistent with the observations in \cite{kunis2020a}, it is not solely the distance between the nodes that determines the spectral properties of the Vandermonde matrix, at least in the near-colliding limit, but also the algebraic variety on which these nodes are situated.

     \par \textbf{Kernel matrices in the flat limit.} Following the previous discussion, let us consider the Gramian matrix {$D_{N}:=\frac{1}{{(2N)}^d}U^*U$}, which is the ``kernel matrix'' for the multidimensional Dirichlet kernel. Kernel matrices are of importance in various fields, including scattered data approximation and machine learning, see e.g. \cite{usevich2021,wathen2015spectral,usevich2024computing} and references therein. In the context of super-resolution, if all the nodes $\mathcal{X}=\{x_j\}_{j=1}^n$ form a single cluster and $\srf\gg 1$, then the kernel matrix $D_{N}$ can be considered in the so-called ``flat limit'', a term introduced in \cite{driscoll2002} in the context of radial basis function (RBF) interpolation. The flat limit was investigated recently in various publications, from among those \cite{usevich2021} and \cite{wathen2015spectral} being most relevant to our work. While \cite{wathen2015spectral} deals exclusively with RBF kernels and derives asymptotic eigenvalue scaling, \cite{usevich2021} extends those results to arbitrary smooth kernels, albeit under the assumption of the points not lying on any low-dimensional algebraic variety. The approach from \cite{wathen2015spectral} is primarily based on Micchelli's lemma \cite{micchelli1986}, and we have adopted this approach previously in the one-dimensional case \cite{batenkov2021}.
    
    \par \textbf{Contributions.} The main result of this paper, Theorem~\ref{Thm:main}, derives the asymptotic scaling of eigenvalues of a kernel matrix induced by an infinitely smooth kernel in the flat limit. {\color{black} This  generalizes} the corresponding result of \cite{usevich2021} to arbitrary geometry, {\color{black}and the result in Theorem 8 of \cite{wathen2015spectral} to arbitrary smooth kernels.}. As a corollary, we derive the  asymptotic scaling of the eigenvalues of $D_{N}$, leading to the appropriate scaling of the singular values of the Vandermonde matrix \eqref{eq:Van}, when the nodes are arranged within a single cluster. In this context, a "single cluster" is characterized by the condition $N\Delta \ll 1$. We therefore obtain a generalization of the results from \cite{kunis2021} which considered mainly the ``single line'' geometry, as well as the extension of \cite[Theorem 2.3]{batenkov2021} to the multivariate case. See also Remark~\ref{rem:weilin}.
    
    In more detail, suppose that $x_j=\Delta y_j$ with $\epsilon=N\Delta$, and take $\epsilon\to 0$. We show that the number of eigenvalues of $D_{N}(\{x_1,\dots,x_n\})$ decaying like $\epsilon^{2k}$, is precisely equal to a certain number $t_k$ which depends on the algebraic properties of the original nodes $\{y_j\}$ (see Theorem \ref{thm:sr}). In particular, when  all the nodes belong to a one-dimensional affine subspace (a ``single-line'' geometry), then $t_k=1$, which implies $\sigma_{\min}(U)\approx N^{\frac{d}{2}}\epsilon^{n-1}$ (as derived in \cite{kunis2021}). On the other hand, when $\{y_j\}$ are in a general position, then, as long as $n\geq {k+d\choose d}$, we have $t_k={k+d-1\choose d-1}$, where $d$ is the dimension of the space. These numbers coincide for $d=1$, but may differ drastically for $d>1$. For example, if $d=2$ then $\sigma_{\min}(U)$ may be as large as $N\epsilon^{\mathcal{O}(\sqrt{n})}$. Our results are slightly more general and provide the eigenvalue scaling for arbitrary (symmetric) sampling sets $\omega\in\mathcal{S}$, with the upper bounds the same as described above; however we can prove matching lower bounds only for certain $\mathcal{S}$ \sout{satisfying} {\color{black}that either satisfy the} the so-called \emph{Geometric Characterization Condition} (in particular satisfied by $\Gg_N$), see Definition~\ref{def:gc} below {\color{black}or are 'lower sets' as described in Definition \ref{def:lower-set}}.
   
   To demonstrate the relevance of our findings to the multivariate SR problem, in Section~\ref{sec:numerics} we also report on a numerical study of reconstructing the signals from noisy measurements using the nonlinear least squares (NLS) method and the multidimensional ESPRIT method \cite{sahnoun2017}. Our results indeed suggest that the asymptotic conditioning of the problem may vary from the ``worst-case" single-line scenario (previously considered in \cite{liu2021}), to the ``best-case'' general position scenario. We are therefore confident that our findings will enhance the understanding of super-resolution stability in high-dimensional settings and can contribute to the analysis of current multivariate SR recovery methods such as \cite{diederichs2022,pena2024,sauer2017,sauer2018, sahnoun2017}. Furthermore, we believe these results can serve as an important step towards establishing min-max error bounds in the general multidimensional scenario, when the nodes form multiple clusters (where some kind of multidimensional ``confluent'' Vandermonde matrices may be required, cf. \cite{batenkov2013a, batenkov2013c, batenkov2018-genspikes, batenkov2023}). 

   {\bf Organization of the paper.} In Section~\ref{sec:prelim} we establish some notation and definitions. In Section~\ref{sec:main-results} we present our main result, Theorem~\ref{Thm:main}, which describes the asymptotic scaling of the eigenvalues of kernel matrices in the flat limit. In Section~\ref{sec:dirichlet} we specialize the results to the Dirichlet kernel matrix, showing, in particular, that the bounds of Theorem~\ref{Thm:main} are tight. In Section~\ref{sec:vandermonde} we derive the corresponding scaling of the singular values of the Vandermonde matrix, which requires additional technical arguments. Finally, numerical results are presented in Section~\ref{sec:numerics}.

    {\bf Acknowledgements.} We thank Konstantin Usevich, Benedikt Diederichs, and Weilin Li for helpful discussions. This work has been supported by the Israel Science Foundation, grant 1793/20, by a collaborative grant from the Volkswagen Foundation, and by the Ariane de Rothschild Women Doctoral Program.

\section{Preliminaries}\label{sec:prelim}
We adapt some definitions and notation from \cite{usevich2021}.
For a multi-index $\boldsymbol{\alpha}=(\alpha_1,\dots,\alpha_d) \in \ZZ^d_{+}$, denote
\[
    \big|\boldsymbol{\alpha}\big| := \sum_{j=1}^d\alpha_j, \quad \boldsymbol{\alpha}! := \alpha_1!\cdot \cdot \cdot \alpha_d!,
\]
and let for each $j=0,1,2,\dots$
\[
    \PP_j := \{\boldsymbol{\alpha}\in \ZZ^d_{+} : \big|\boldsymbol{\alpha}\big| \leq j\}, \quad
    \HH_j := \{\boldsymbol{\alpha}\in \ZZ^d_{+} : \big|\boldsymbol{\alpha}\big| = j\} = \PP_j\setminus\PP_{j-1},\quad \PP_{-1}:=\emptyset.
\]
The cardinalities of these sets are given by:
\begin{equation}\label{eq:card}
    p_j := \#\PP_j = {j+d \choose d}, \quad h_j := \#\HH_j = {j+d-1 \choose d-1} = \#\PP_j - \#\PP_{j-1}.
\end{equation}

For $\boldsymbol{x} \in \RR^d$ and $\boldsymbol{\alpha}\in \ZZ^d_{+}$,  the monomial $\boldsymbol{x}^{\boldsymbol{\alpha}}$ is defined as $\boldsymbol{x}^{\boldsymbol{\alpha}}:= x_1^{\alpha_1} \cdot \cdot \cdot x_d^{\alpha_d}$,  and its (total) degree is $|\boldsymbol{\alpha}|$. Let $\prec$ be a fixed graded order on the set of multi-indices $\ZZ^d_{+}$ (or, equivalently, on the set of monomials), for example a  graded lexicographic order or a graded reverse lexicographic order. {\color{black}The ordering $\prec$ is used throughout the sequel.}

% {\color{gray}
% The relation $\prec$ defines a lexicographic graded order for all pairs of multi-indices, i.e.
% \[
%     \big|\boldsymbol{\alpha}\big| < \big|\boldsymbol{\beta}\big| \quad \iff \quad \boldsymbol{\alpha} \prec \boldsymbol{\beta}.
% \]
% }

\begin{definition}\label{smooth-kernel}
    Let $\Omega\subseteq \RR^d$ be a fixed open set. A kernel function $\mathcal{K}:\Omega \times \Omega \rightarrow \RR$ is said to belong to the class $\mathcal{C}^{(r,r)}(\Omega)$ (or $\mathcal{C}^{\PP_r\times \PP_r}(\Omega)$) when the partial derivatives $$\mathcal{K}^{(\boldsymbol{\alpha,\beta})}(\boldsymbol{x},\boldsymbol{y}):=\frac{\partial^{|\boldsymbol{\alpha}|+|\boldsymbol{\beta}|}}{\partial x_1^{\alpha_1}\cdot\cdot\cdot\partial x_d^{\alpha_d}\partial y_1^{\beta_1}\cdot\cdot\cdot\partial y_d^{\beta_d}}\mathcal{K}(\boldsymbol{x},\boldsymbol{y})$$ exist and are continuous on $\Omega \times\Omega$, for all $\boldsymbol{\alpha,\beta}\in \PP_r$. When $\mathcal{K}\in\mathcal{C}^{(r,r)}(\Omega)$ for all $r\in\NN$, we write $\mathcal{K}\in\mathcal{C}^{(\infty,\infty)}(\Omega)$.
\end{definition}

\begin{definition}
    For a kernel $\mathcal{K} \in \mathcal{C}^{(r,r)}(\Omega)$, a finite set $\ci {\color{black}= \{\boldsymbol{x}_1,\dots,\boldsymbol{x}_n\}} \subset \Omega$ and $\epsilon > 0 $ we define the scaled kernel $\mathcal{K}_{\epsilon}(\boldsymbol{x},\boldsymbol{y}):=\mathcal{K}(\epsilon \boldsymbol{x}, \epsilon \boldsymbol{y})$ and the kernel matrix
    \begin{equation}
        K_{\epsilon}(\ci) := \big[ \mathcal{K}_{\epsilon}(\boldsymbol{x}, \boldsymbol{x'})\big]_{\boldsymbol{x},\boldsymbol{x'} \in \ci} \in \RR^{\color{black}n\times n}.
    \end{equation}
\end{definition}

\begin{definition}[Multivariate Vandermonde matrix]
For a finite set $\ci=\{\boldsymbol{x}_1,\dots,\boldsymbol{x}_n\}\subset\RR^d$ of nodes, and an $\prec$-ordered set of multi-indices $\mathcal{A}=\{\boldsymbol{\alpha}_1,\dots,\boldsymbol{\alpha}_m\} \subset \ZZ^d_{+}$ (i.e. $\boldsymbol{\alpha}_1\prec \dots \prec \boldsymbol{\alpha}_m$), we define the multivariate Vandermonde matrix as
\begin{equation}\label{eq:multiVan}
    V_{\mathcal{A}} = V_{\mathcal{A}}(\ci) = \big[(\boldsymbol{x}_i)^{\boldsymbol{\alpha}_j}\big]_{1\leq i\leq n}^{{1\leq j\leq m}}\in\RR^{n\times m}.
\end{equation}
\end{definition}

{\color{black}Note that \(V\) is a multivariate real Vandermonde matrix with monomial entries, in contrast to the complex Vandermonde matrix \(U\) defined in \eqref{eq:Van}, whose entries are of the form \(e^{\imath\langle\omega,x\rangle}\).}
Let us further denote $V_{\leq k} := V_{\PP_k}(\ci)$ and $V_{k} := V_{\HH_k}(\ci)$. The matrix $V_{\leq k} \in \RR^{n \times p_k}$ can be partitioned into $k+1$ matrices $V_k \in \RR^{n \times h_k}$ arranged by increasing degree:
\begin{equation}\label{eq:van-deg-k}
     V_{\leq k} = \big[V_0 \quad V_1 \quad \dots \quad V_k\big].
\end{equation}

\begin{definition}\label{def:wronskian}
    Let $\mathcal{K}\in \mathcal{C}^{(r,r)}(\Omega)$, with $\Omega\subseteq \RR^d$ and $\mathbf{0}\in\Omega$. Let $\mathcal{A},\mathcal{B}\subset \ZZ^d_+$ be two sets of multi-indices satisfying $|\alpha|,|\beta|\leq r$ for all $\alpha \in \mathcal{A},\beta \in \mathcal{B}$. The Wronskian matrix of $\mathcal{K}$ is defined as:
    \begin{equation}\label{eq:wronskian}
        W^{\mathcal{K}}_{\mathcal{A},\mathcal{B}} = \bigg[{{\mathcal{K}^{(\alpha,\beta)}(0,0)}\over{\alpha! \beta!}}\bigg]_{\alpha \in \mathcal{A},\beta \in \mathcal{B}},
    \end{equation}
    % \begin{equation}\label{eq:wronskian-dirichlet}
    %  W^{\mathcal{K}}_{\mathcal{A},\mathcal{B}} = \bigg[{{D_{\sss}^{(\alpha,\beta)}(0,0)}\over{\alpha! \beta!}}\bigg]_{\alpha \in \mathcal{A},\beta \in \mathcal{B}},
    % \end{equation}
where the rows and columns are indexed by multi-indices in $\mathcal{A}$ and $\mathcal{B}$ according to the chosen ordering. In addition, denote by $W^{\mathcal{K}}_{\leq k}=W^{\mathcal{K}}_{\PP_k}=W^{\mathcal{K}}_{\PP_k,\PP_k}$.
\end{definition}

%(define $W_{\leq k}=W_{\PP_k}(\ci)=W_{\PP_k,\PP_k}(\ci)$ - Taylor coefficients matrix)\\
The final piece of notation is central to our discussion, describing a geometric property of $\ci$.

\begin{definition}[Discrete moment order, \cite{schabak2005}]\label{def:mu-def}
The \emph{discrete moment order} of a set $\ci$ consisting of $n$ pairwise distinct points, is the smallest number $m=\mu(\mathcal{X})$ such that
$$
\ker(V_{\leq m}(\ci)^T)=\emptyset 
$$
{\color{black}holds, or equivalently, $$\rank (V_{\leq m}(\ci))=n$$ holds.}
\end{definition}

Note that such $m$ exists (is finite). Indeed, if $\ci=\{x_1,\dots,x_n\}$, then $\mu(\ci)\leq n-1$ because the following Lagrange interpolation polynomials, each of total degree $n-1$, are linearly independent:
$$
\biggl\{\ell_i(x)=\prod_{j\neq i}\frac{H_j(x)}{H_j(x_i)}\biggr\}_{i=1}^n,
$$
where $H_j$ is some affine hyperplane containing $x_j$ and not containing every other $x_k,k\neq j$.
% because we always have that $\ker(V_{\leq n-1}(\ci))=\emptyset$ for any set $\ci$ consisting of $n$ different points.
 Cf.  \cite[Definition 3]{schabak2005},\cite{GascaMariano2000}  and \cite{wathen2015spectral}.

\section{Eigenvalues of smooth kernel matrices in the flat limit}\label{sec:main-results}

In this section, we consider the kernel matrix $K_{\epsilon}=K_{\epsilon}(\ci)$ for a fixed infinitely smooth kernel $\mathcal{K}$ and a fixed set $\ci\subset\RR^d$ of $n$ pairwise distinct points. Our main interest is in the asymptotic order of the decay of the eigenvalues of $K_{\epsilon}$ in the flat limit $\epsilon\to 0$. We employ the standard $O(\cdot)$ notation.

Our main result gives an upper bound on the decay rates of all the eigenvalues.

\begin{theorem}\label{Thm:main}
    Let $m=\mu(\ci)$ be the discrete moment order of the set $\ci{\color{black}\subset\RR^d}$. For any infinitely smooth kernel $\mathcal{K}\in \mathcal{C^{(\infty,\infty)}}(\Omega)$ with $\mathbf{0}\in\Omega$, and small enough $\epsilon$, the eigenvalues of $K_{\epsilon}(\ci)$ split into $m+1$ groups
    \begin{align*}
        \lambda_{0,0} &= O(1), \quad \{\lambda_{1,j}\}_{j=1}^{t_1}=O(\epsilon^2),\quad \dots \quad, \quad\{\lambda_{m,j}\}_{j=1}^{t_m}=O(\epsilon^{2m}),\\
        t_k&:=\rank(V_{\leq k}(\ci))-\rank(V_{\leq k-1}(\ci)),\;k=1,2,\dots,m.
    \end{align*}
    \begin{remark}
        Theorem \eqref{Thm:main} is valid also for kernels in $\mathcal{C}^{(m+1,m+1)}$.
    \end{remark}
\end{theorem}
\begin{proof}
    We start by writing the Taylor expansion of $K_{\epsilon}$ around 0 as equation (55) in \cite{usevich2021}: 
    \begin{equation}\label{eq:taylor}
        K_\epsilon = V_{\leq m}\Delta_m W \Delta_m V_{\leq m}^T + \epsilon^{m+1}(V_{\leq m}\Delta_m W_1(\epsilon)+W_2(\epsilon)\Delta_m V_{\leq m}^T)+\epsilon^{2{(m+1)}}W_3(\epsilon) 
    \end{equation}
    where $W_i(\epsilon) = O(1)$, $W = W^{\mathcal{K}}_{\leq m}$ and $\Delta_m(\epsilon)=\diag(1,\epsilon I_{h_1},\dots,\epsilon^m I_{h_m}) \in \RR^{p_m \times p_m}$. 
    %Here we choose $k$ such that $\#\mathbb{P}_k = n$.
    For the full derivation of \eqref{eq:taylor}, see section \ref{app:taylor} in the Appendix.

    Let $V = V_{\leq m}$ where $V = [V_0 \dots V_{m}] = [V_{\leq m-1} \quad V_m]$, and $t_k := \rank(V_{\leq k})-\rank(V_{\leq k-1})$ as defined in the theorem.
    In \cite{usevich2021} only the generic case $t_k=h_k-h_{k-1}$ is considered.  Here we generalize Theorem 6.1 in \cite{usevich2021} to the case where the multivariate Vandermonde matrices $V_{\leq k}$ are not necessarily full rank.

    Consider the full QR decomposition of $V_{\leq m-1}$ with column pivoting, as elaborated in Section 5.4.2, specifically, formula (5.4.6) in the book \cite{golub2013}: 
\[
    V_{\leq m-1} P_{m-1} = [Q_{\leq m-1} \quad Q_{\perp}] \begin{bmatrix}R_{11} & R_{12} \\ 0 & 0\end{bmatrix} = Q\begin{bmatrix}
        R \\
        0
    \end{bmatrix}
\]
where $Q_{\leq m-1} \in \RR^{n\times r_{m-1}}$, $Q_{\perp} \in \RR^{n\times t_m}$, $R_{11} \in \RR^{r_{m-1}\times r_{m-1}}$, $R \in \RR^{r_{m-1}\times p_{m-1}}$, $r_{m-1}+t_m=n$ and $P_{m-1}$ is a permutation matrix. \\
%\[
 %   V = \big[Q_0R_0 \dots Q_{k}R_{k}\big] = [Q_0 \dots Q_{k}]\cdot blkdiag(R_0,\dots,R_{k}) =: QR
%\]
We have $[Q_{\leq m-1} \quad Q_{\perp}]^TV_{\leq m}P = \begin{bmatrix}R & Q_{\leq m-1}^TV_m \\ 0 & Q_{\perp}^TV_m\end{bmatrix}$
and $Q_{\perp}^TV_{\leq m}P = [0 \quad Q_{\perp}^TV_{m}]$ where $P:= \begin{bmatrix} P_{m-1} & 0 \\ 0 & I_{h_m}\end{bmatrix}$. Note that for every permutation matrix, we have $P^{-1}=P^T$. \\

%\begin{proof}
%    show this for elementary permutation matrix...
%\end{proof}
Thus we can rewrite the expansion \eqref{eq:taylor} as follows:
\begin{align*}
    K_\epsilon &= \quad V_{\leq m}PP^T\Delta_mPP^T WPP^T \Delta_mPP^T V_{\leq m}^T\\
    &\quad  + \epsilon^{m+1}(V_{\leq m}PP^T\Delta_mPP^T W_1(\epsilon)+W_2(\epsilon)PP^T\Delta_mPP^T V_{\leq m}^T)\\
    &\quad +\epsilon^{2{(m+1)}}W_3(\epsilon) \\
    K_\epsilon &= \tilde{V}_{\leq m}\tilde{\Delta}_m\tilde{W}\tilde{\Delta}_m\tilde{V}_{\leq m}^T + \epsilon^{m+1}(\tilde{V}_{\leq m}\tilde{\Delta}_m \tilde{W}_1(\epsilon)+\tilde{W}_2(\epsilon)\tilde{\Delta}_m\tilde{V}_{\leq m}^T)+\epsilon^{2{(m+1)}}W_3(\epsilon) 
\end{align*}
where $\tilde{V}_{\leq m} := V_{\leq m}P$ , $\tilde{\Delta}_m := P^T\Delta_mP = {\color{black}\diag(\epsilon^{i_1},...,\epsilon^{i_{q_{m-1}}},\epsilon^{m}I_{h_m})}${\color{black}, $q_{m-1}:=\sum_{i=0}^{m-1}h_{i}$} due to Lemma \ref{app:lem}, $\tilde{W} := P^TWP$, $\tilde{W}_1(\epsilon):=P^TW_1(\epsilon)$ and $\tilde{W}_2(\epsilon):=W_2(\epsilon)P$. \\
Multiplying by $Q_{\perp}^T$ and its transpose we get:
\begin{align*} 
    Q_{\perp}^TK_{\epsilon}Q_{\perp} &= \quad[0 \quad Q_{\perp}^TV_{m}]\tilde{\Delta}_{\color{black}m} \tilde{W} \tilde{\Delta}_{\color{black}m}[0 \quad Q_{\perp}^TV_{m}]^T\\
    &\quad + \epsilon^{m+1}[0 \quad Q_{\perp}^TV_{m}]\tilde{\Delta}_{\color{black}m} \tilde{W}_1Q_{\perp}\\
    &\quad + \epsilon^{m+1}Q_{\perp}^T\tilde{W}_2\tilde{\Delta}_{\color{black}m} [0 \quad Q_{\perp}^TV_{m}]^T\\
    &\quad + (\epsilon^{m+1}Q_{\perp}^TW_3Q_{\perp}\epsilon^{m+1}).
\end{align*}
%\[ 
 %   \Delta_t^{-1}Q^TK_{\epsilon}Q\Delta_t^{-1} = RWR + \epsilon^nRW_1Q\Delta_t^{-1} + \epsilon^n\Delta_t^{-1}Q^TW_2R^T + (\epsilon^n\Delta_t^{-1})QTW_3Q(\epsilon^n\Delta_t^{-1}) \\
  %  = RWR^T + O(\epsilon)
%\]
We can also write the above as:
\begin{align*}
    Q_{\perp}^TK_{\epsilon}Q_{\perp} &=\quad O(\epsilon^{2m})Q_{\perp}^TV_{m}\hat{W}V_{m}^TQ_{\perp}\\
    &\quad + O(\epsilon^{2m+1})(Q_{\perp}^TV_{m}\hat{W_1}Q_{\perp}\\
    &\quad  + Q^T_{\perp}\hat{W_2}V_{m}Q_{\perp})\\
    &\quad + O(\epsilon^{2{(m+1)}}),
\end{align*}
where $\hat{W}$ is the lower right $h_m \times h_m$ sub-matrix of $\tilde{W}$. Thus, by the Courant-Fischer principle, $K_{\epsilon}$ has at least $t_m = \dim(Q_{\perp})$ eigenvalues which decay as fast as $O(\epsilon^{2m})$. 
\par For any $k=1,...,m-1$, let $V_{\leq k-1}P_{k-1}=[Q_{\leq k-1}\quad Q_{\perp, k-1}] \begin{bmatrix}R_{\leq k-1} \\ 0 \end{bmatrix}$ be the full QR decomposition of $V_{\leq k}$ with pivoting, where $Q_{\leq k-1} \in \RR^{n\times [n-(t_k+\dots+t_m)]}$, $Q_{\perp, k-1} \in \RR^{n\times [t_k+\dots+t_m]}$ and $R_{\leq k-1} \in \RR^{[n-(t_k+\dots+t_m)] \times p_{k-1}}$. 

We can apply the same argument on $Q_{\perp, k-1}$ and $V_{\leq k-1}$ by induction, then we get that there are at least $t_k+...+t_m$ eigenvalues which decay as fast as $O(\epsilon^{2k})$ and there are at least $t_{k+1}+...+t_m$ eigenvalues which decay as fast as $O(\epsilon^{2({k+1})})$. Therefore, there are at least $t_k$ eigenvalues which decay as fast as $O(\epsilon^{2{k}})$. Since $\sum_{i=0}^{m}t_i=n$ (assume $t_{0}=0$) we get that there are \sout{exactly} $t_k$ eigenvalues decaying (at least) like $\epsilon^{2k}$ when $\epsilon$ goes to 0.
\end{proof}
%\[ Q^TK_{\epsilon}Q = \Delta_m (RWR^T+O(\epsilon)) \Delta_m = \begin{bmatrix}
%                                                                    O(1) & \epsilon C_{1\times t_1} & \dots & \epsilon^mT_{1\times t_m} \\
%                                                                    \epsilon C_{t_1\times 1} & O(\epsilon^2)I_{t_1} & \dots & O(\epsilon^m)I_{t_m} \\
%                                                                    \vdots & & \ddots &  \vdots \\
%                                                                    \epsilon^m C_{t_m\times 1} & \epsilon^{m+1} C_{t_m\times t_1} & \dots & O(\epsilon^{2m})I_{t_m}
%                                                              \end{bmatrix}
%\]
%The orders of the eigenvalues are on the diagonal.

The bounds of Theorem~\ref{Thm:main} are not necessarily tight. As we show in the next lemma, the exactness of the asymptotic orders depends on the analytic properties of the kernel $\mathcal{K}$ and the geometry of $\ci$, therefore additional assumptions ought to be made on both of those objects in order to obtain matching lower bounds.

{\color{black}
\begin{definition}\label{def:rtilde}
%By a slight abuse of notation, w
    Let $\ci\subset\RR^d$. Consider the matrices $Q_{\perp, k-1}$ defined in the proof of Theorem~\ref{Thm:main}.
    {\color{black}We define a sequence $\{\mathcal{M}_k\}_{k=0}^m$ of linear subspaces of $\mathbb{R}^n$ with $\dim\mathcal{M}_k=t_k$, $\mathcal{M}_{k}\subseteq \operatorname{range}(Q_{\perp, k-1})$ and $ \mathcal{M}_{k} \perp \operatorname{range}(Q_{\perp, k})$. By a slight abuse of notation, we also denote by $\mathcal{M}_{k} \in \RR^{n \times t_k}$ a matrix whose columns span $\mathcal{M}_{k}$. For consistency we set  $\operatorname{range}(Q_{\perp,-1})\equiv\RR^{n}$. Since $\sum_{i=0}^m t_i=n$, we have
    \[ \mathcal{Q}_{\perp} := \big(\mathcal{M}_{0}\;\mathcal{M}_{1}\;\dots\mathcal{M}_{m}\big) \in \RR^{n\times n}.\]
    }
    %we take the full QR decomposition of $V_{\leq m} = QR$ where $Q$ is of size $n \times n$ then by Lemma 6.4 in \cite{usevich2021} we have
    The matrix $\mathcal{Q}_{\perp}^TV_{\leq m}$ is upper triangluar. We denote its block-diagonal part by
    $$
    \tilde{R}:=\operatorname{blkdiag} \left(\mathcal{Q}_{\perp}^TV_{\leq m}\right) = \diag\bigl\{\mathcal{M}_{0}^TV_0,...,\mathcal{M}_{m}^TV_m\bigr\}.
    $$
    Each diagonal block of $\tilde{R}$ is full rank by definition, therefore $\tilde{R}$ is full rank.
\end{definition}

\begin{lemma}\label{lem:det-K}
    For $\mathcal{K} \in \mathcal{C^{(\infty,\infty)}}(\Omega)$ and $\ci\subset\RR^d$, let $\tilde{R}\in\RR^{n\times p_m}$  be as in Definition~\ref{def:rtilde}. Define $C=\det\left(\tilde{R} W \tilde{R}^T\right)\in\mathbb{R}$ where $W=W^{\mathcal{K}}_{\leq m}\in\RR^{p_m\times p_m}$ is the Wronskian matrix of $\mathcal{K}$ as in Definition~\ref{def:wronskian}. Then 
    \[ 
        \det(K_{\epsilon}(\ci)) = \epsilon^{2\sum_{i=0}^{m}it_i}(C + O(\epsilon)), \quad \epsilon\to 0.
    \]
\end{lemma} 
\begin{proof}
    By \cite[Lemma 6.4]{usevich2021} we have
    \[ E_n^{-1}\mathcal{Q}_{\perp}^TV_{\leq m}\Delta_m = \tilde{R}+O(\epsilon)\]
    where $E_n = \diag\{1,\epsilon I_{t_1},...,\epsilon^m I_{t_m}\} \in \RR^{n\times n}$ and $\tilde{R} \in \RR^{n\times p_m}$. \\
    Thus, using \eqref{eq:taylor}, we get that
    \begin{align*}
        E_n^{-1}\mathcal{Q}_{\perp}^TK_{\epsilon}\mathcal{Q}_{\perp}E_n^{-1} &= (\tilde{R}+O(\epsilon))W(\tilde{R}^T+O(\epsilon))\\ &+ \epsilon^{m+1}((\tilde{R}+O(\epsilon)) W_1(\epsilon)\mathcal{Q}_{\perp}E_n^{-1}\\ &+ E_n^{-1}\mathcal{Q}_{\perp}^TW_2(\epsilon)(\tilde{R}^T+O(\epsilon)))\\ &+\epsilon^{2{(m+1)}}E_n^{-1}\mathcal{Q}_{\perp}^TW_3(\epsilon)\mathcal{Q}_{\perp}E_n^{-1} \\ &= \tilde{R}W\tilde{R}^T + O(\epsilon)
    \end{align*}
     where the last equality follows from $\epsilon^{2{(m+1)}}E_n^{-1} = O(\epsilon)$ and $\tilde{R}W\tilde{R}^T \in \RR^{n\times n}$ .
     This implies that $\epsilon^{-2\sum_{j=0}^mjt_j}\det(K_{\epsilon})=\det(\tilde{R}W\tilde{R}^T+O(\epsilon)) = \det(\tilde{R}W\tilde{R}^T)+O(\epsilon)$ (by Lemma \eqref{app:det-lem} in the Appendix) with $C = \det(\tilde{R}W\tilde{R}^T)$.
     
\end{proof}
}

%\todo[inline]{Check. In general the eigenvalues need not be analytic functions of $\epsilon$ in the neighborhood of $\epsilon=0$ (e.g. due to multiplicity), but for Hermitian matrices, they are.}
\begin{corollary}\label{cor:tightness}
    Suppose $\mathcal{K}$ is symmetric and analytic (in both variables) in the neighborhood of $(\mathbf{0},\mathbf{0})$. The scaling of the eigenvalues as given in Theorem~\ref{Thm:main} is exact, if and only if $\det(\tilde{R}W\tilde{R}^T) \neq 0$. {\color{black}By 'exact scaling' we mean that there exists a constant $\tilde{\lambda}_{k,j} \neq 0$ independent of $\epsilon$ such that $\lambda_{k,j}=\epsilon^{2k}(\tilde{\lambda}_{k,j}+O(\epsilon))$}.
\end{corollary}
\begin{proof}
By Theorem~\ref{Thm:main} and \cite[Theorem 2.9]{usevich2021}, the eigenvalues of $K_{\epsilon}(\ci)$ satisfy
\[
    \lambda_{0,0} = \epsilon^{0}(\tilde{\lambda}_{0,0}+O(\epsilon)), \quad \{\lambda_{1,j}\}_{j=1}^{t_1}=\{\epsilon^{2}(\tilde{\lambda}_{1,j}+O(\epsilon))\}_{j=1}^{t_1}, ... ,\{\lambda_{m,j}\}_{j=1}^{t_m}=\{\epsilon^{2m}(\tilde{\lambda}_{m,j}+O(\epsilon))\}_{j=1}^{t_m},
\]
 where $t_k:=\rank(V_{\leq k})-\rank(V_{\leq k-1})$ and $m=\mu(\ci)$, and $\tilde{\lambda}_{s,j}$ do not depend on $\epsilon$ (but may be zero). By Lemma~\ref{lem:det-K} we have that
$$
\big|K_{\epsilon}(\ci)\big|=\epsilon^{2\sum_{i=0}^{m}it_i}\biggl\{\big|\tilde{R}W\tilde{R}^T\big| + O(\epsilon)\biggr\}.
$$

On the other hand, the determinant is the product of eigenvalues, therefore
\begin{align*}
\big|K_{\epsilon}(\ci)\big| &= \prod_{j=1}^n\lambda_j = \lambda_{0,0}\prod_{j=1}^{t_1}\lambda_{1,j}...\prod_{j=1}^{t_m}\lambda_{m,j} = (\tilde{\lambda}_{0,0}+O(\epsilon))\prod_{j=1}^{t_1}\epsilon^{2}(\tilde{\lambda}_{1,j}+O(\epsilon))...\prod_{j=1}^{t_m}\epsilon^{2m}(\tilde{\lambda}_{m,j}+O(\epsilon))\\
&= \epsilon^{2\sum_{i=0}^{m}it_i} \biggl\{\prod_{s=1}^m \prod_{j=1}^{t_s} \tilde{\lambda}_{s,j}+O(\epsilon)\biggr\}.
\end{align*}

Now we just compare the two expressions above. In one direction, if the scaling is exact then $\tilde{\lambda}_{s,j} \neq 0$ for all $0 \leq s \leq m$ and $0 \leq j \leq t_s$, which implies $\left|\tilde{R}W\tilde{R}^T\right|\neq 0$. In the other direction, if there are $\ell \geq 1$ zeros $\tilde{\lambda}_{k} \in \{\tilde{\lambda}_{s,j}\}_{s=1,\dots,m}^{j=1,\dots,t_m}$ and $\tilde{\lambda}_{k}=0$ for $k=1,...,\ell$, then $\big|K_{\epsilon}\big| = \epsilon^{\ell+2\sum_{i=0}^{m}it_i}(\prod_{k=\ell+1}^{n}\tilde{\lambda}_{k}+O(\epsilon))$, forcing $\left|\tilde{R}W\tilde{R}^T\right|=0$.
\end{proof}

\section{Eigenvalues of the multidimensional Dirichlet kernel}\label{sec:dirichlet}

In this section we consider the specific case of the multidimensional Dirichlet kernel, which is directly motivated by the problem of super-resolution as outlined in the Introduction. Let $\sss \subset \RR^d$ denote a finite set of frequencies,  symmetric around the origin, where the spectral data \eqref{eq:samples} is available. We first derive a generic sufficient condition on $\sss$ which ensures that the scaling order of Theorem~\ref{Thm:main} is exact. Later  we show that this condition is satisfied by $\sss=\Gg_N$, the regular symmetric grid in $\RR^d$ of side length $2N+1$, for large enough $N$.

\begin{definition}
    For $\sss \subset \RR^d$, symmetric around $\mathbf{0}\in\RR^d$, the (real-valued) Dirichlet kernel $\mathcal{D}_{\sss}$ is defined as follows:
    $$
    \mathcal{D}_{\sss}(x)=\sum_{\omega\in\sss} \exp\left( \imath \langle x,\omega \rangle \right).
    $$
    By a slight abuse of notation, we also denote by $\mathcal{D}_{\sss}$ the corresponding bi-multivariate kernel:
    \begin{equation}\label{eq:dirichlet-nd-def}
    \mathcal{D}_{\sss}(x,y)=\mathcal{D}_{\sss}(x-y)=\sum_{\omega\in\sss} \exp\left( \imath \langle x-y,\omega \rangle \right).
    \end{equation}
    Given $\ci\subset\RR^d$ and $\epsilon>0$, the corresponding Dirichlet kernel matrix and its flat limit version are denoted as:
    \begin{align}\label{eq:Dk}
        \begin{split}
            D(\ci,\sss) &:= \big[\mathcal{D}_{\sss}(x,x')\big]_{x,x'\in\ci},\\
         D_{\epsilon}(\ci, \sss) &:= D(\epsilon\ci, \sss) = \big[ \mathcal{D}_{\sss}(\epsilon x,\epsilon x')\big]_{x,x' \in \ci}.
        \end{split}
    \end{align}
    In addition, if $\sss=\sss_1\times...\times \sss_d$ is a tensor product grid, we have $\mathcal{D}_{\sss}(x,y)=\prod_{j=1}^d\mathcal{D}_{\sss_j}(x_j,y_j)$.
\end{definition}

\subsection{A nondegeneracy condition}

Note that $\mathcal{D}_{\sss}$ is a translation invariant, symmetric and analytic kernel. Recall the definition \eqref{eq:wronskian} of the Wronskian matrix for an arbitrary kernel. Set $W:=W^{\mathcal{D_S}}_{\leq m}$. By Corollary~\ref{cor:tightness}, in order to show tightness, it is sufficient to demonstrate that $C := \det(\tilde{R}W\Tilde{R}^T) \neq 0$. In the following lemma, we show that this holds under a certain (more or less natural) non-degeneracy assumption on the set $\sss$.

%Need to show that $\det(W)$ isn't always zero (for $n=d=2$ with samples on the grid was shown to be nonzero - wolfram mathematica). Try proving by induction:\\ assume $\det(W_{\leq m-1})\neq 0$ then $\det(W_{\leq m})=\det(W_{\leq m-1})\det(W_{22}-W_{21}(W_{\leq m-1})^{-1}W_{12})\neq 0$.

\begin{lemma}\label{lem:det}
    For $W:=W_{\leq m}$ as defined above, $\det(\tilde{R}W\Tilde{R}^T) > 0$ holds whenever $\rank(V_{\leq m}(\sss)) \geq n$, {\color{black}where $V_{\leq m}(\sss) = \big[{(\omega)}^{\alpha}\big]_{\omega \in \sss}^{\alpha \in \PP_m}$.}
\end{lemma}

\begin{remark}
    Note that contrary to Section~\ref{sec:main-results}, the Vandermonde matrix $V_{\leq m}$ is here evaluated on the set $\sss$ (and not on $\ci$), while the polynomial total degree is still $m=\mu(\ci)$.
\end{remark}

\begin{proof}[Proof of Lemma~\ref{lem:det}]

    We first show that $W$ is positive semidefinite. Directly from \eqref{eq:wronskian} and \eqref{eq:dirichlet-nd-def}, we have
    \begin{align*}
        W &= \bigg[{{\mathcal{D}_{\sss}^{(\alpha,\beta)}(0,0)}\over{\alpha! \beta!}}\bigg]_{\alpha,\beta \in \PP_m} = \bigg[\sum_{\omega \in \sss}{{(i\omega)^{\alpha}(-i\omega)^{\beta}}\over{\alpha! \beta!}}\bigg]_{\alpha,\beta \in \PP_m} \\
        &= \bigg[\sum_{\omega =(\omega_1,\dots,\omega_d)\in \sss}{{(i)^{|\alpha|}(-i)^{|\beta|}\omega_1^{\alpha_1+\beta_1}\cdot ... \cdot \omega_d^{\alpha_d+\beta_d}}\over{\alpha! \beta!}}\bigg]_{\alpha,\beta \in \PP_m}\\
         &= FV^T_{\leq m}(\sss)V_{\leq m}(\sss)F^{*},
     \end{align*}
     where $F := \diag\left(\frac{i^{|\alpha|}}{\alpha !}\right)_{\alpha \in \PP_m}$. Let $B := F V^T_{\leq m}(\sss)$, and consider $B=U\Sigma T^*$ the singular value decomposition of $B$. We therefore have $W = BB^* = U\Sigma^2 U^*$, implying that $W$ is positive semidefinite (since $W$ is Hermitian and has non-negative eigenvalues). Since $\tilde{R}$ is full rank {\color{black} (recall Definition~\ref{def:rtilde})}, it follows that $\tilde{R}W\tilde{R}^T$ is positive semidefinite as well, implying that $\det(\tilde{R}W\tilde{R}^T) \geq 0$. To ensure $\det(\tilde{R}W\tilde{R}^T) > 0$, we require that $\tilde{R}W\tilde{R}^T$ be positive definite. The latter condition is satisfied when $\rank(\tilde{R}W\tilde{R}^T) \geq n$, which is equivalent to $\rank(W) = \rank(V_{\leq m}(\sss)) \geq n$.
\end{proof}
\begin{remark}
    For $d=1$, we have that $m=n-1$ and $p_{n-1} = {n \choose 1} = n$, meaning that the condition $\rank(V_{\leq n-1}(\sss)) \geq n$ in Lemma~\ref{lem:det} is satisfied as long as the sampling set $\sss$ contains at least $n$ distinct points.
\end{remark}

 By applying Corollary~\ref{cor:tightness}, we immediately obtain the following result.

\begin{lemma}\label{lem:exact-scaling}
    Assume that the sampling set $\sss$ satisfies the condition $\rank(V_{\leq m}(\sss)) \geq n$. Then the eigenvalues of $D_{\epsilon}(\ci,\sss)$ split into $m+1$ groups
    \[
        \lambda_{0,0} = \epsilon^{0}(\tilde{\lambda}_{0,j}+O(\epsilon)), \quad \{\lambda_{1,j}\}_{j=1}^{t_1}=\{\epsilon^{2}(\tilde{\lambda}_{1,j}+O(\epsilon))\}_{j=1}^{t_1}, ... ,\{\lambda_{m,j}\}_{j=1}^{t_m}=\{\epsilon^{2m}(\tilde{\lambda}_{m,j}+O(\epsilon))\}_{j=1}^{t_m},
    \]
     where $t_k:=\rank(V_{\leq k}(\ci))-\rank(V_{\leq k-1}(\ci))$ and $\tilde{\lambda}_{s,j} \neq 0$ for all $0 \leq s \leq m$ and $0 \leq j \leq t_s$.
\end{lemma}

\subsection{Geometric characterization condition and the uniform grid}
The rank condition in Lemma~\ref{lem:det} is nontrivial to verify for a given set $\sss$. In this section we provide an explicit example of such a set, namely, the uniform grid $\Gg_N$ in $\RR^d$. Our approach is a straightforward application of well-known results in multivariate polynomial interpolation, adapted to our setting.

We first recall some definitions and results from \cite{sauer2006}. {\color{black}For a set \(\Theta\), let \(\#\Theta\) denote its cardinality. For a square matrix \(M\), let \(|M|\) denote its determinant, i.e. \(|M|:=\det(M)\).}

\begin{definition}[\cite{sauer2006}, p.194]
    Let $\mathcal{P}$ be a linear subspace of the polynomial ring $\Pi := \RR\left[x_1,...,x_d\right]$. The polynomial interpolation problem with respect to a (finite) set of linearly independent functionals $\Theta \subset \Pi'$ ($\Pi'$ is the dual space {\color{black}of $\Pi$}) is said to be \textbf{poised} for $\mathcal{P}$, if for any $Y \in \RR^{\#\Theta}$ there exists a unique $f \in \mathcal{P}$ such that $\Theta f = Y$.
\end{definition}

\begin{theorem}[\cite{sauer2006} p.194]\label{thm:poised}
    For $\mathcal{P} \subset \Pi$, and a finite set $\Theta \subset \Pi'$ as above, the following statements are equivalent:
    \begin{enumerate}
        \item The polynomial interpolation problem with respect to $\Theta$ is poised for $\mathcal{P}$.
        \item $\dim \mathcal{P} = \#\Theta$ and $\Theta P = \big[ \theta p : \theta \in \Theta , p \in P\big]$ satisfies $|\Theta P| \neq 0$ for any basis $P$ of $\mathcal{P}$.
    \end{enumerate}
\end{theorem}

For $\Theta_{\sss} := \{\delta_{\xi}\}_{\xi \in \sss}$ the set of point evaluation functionals on $\sss$, and the standard monomial basis $P_n$ of polynomials of total degree $\leq n$, we have in particular $\dim P_n=\#\sss={{n+d} \choose d}=p_n$, and the matrix $\Theta_{\sss} P_{\color{black}n} = V := V_{\leq n}(\sss)$ is our Vandermonde matrix. Therefore, if we can find a sample set $\sss$ such that the polynomial interpolation problem with respect to $\Theta_{\sss}$ is poised for {\color{black}$\mathcal{P}_n := \text{span}\{P_n\}$}, this would immediately imply that $|V| \neq 0$.  Examples of such situations are provided in \cite{chung1977, GascaMariano2000}, and it turns out that there is a general sufficient condition which ensures poisedness.

\begin{definition}[\cite{chung1977}]\label{def:gc}
    A lattice $\mathcal{J}:=\{x_1,...,x_N\}$ of $N=p_n$ distinct nodes of $\RR^d$ is said to satisfy the \textbf{Geometric Characterization Condition (GC)} (of degree $n$) if corresponding to each node $x_i$ there exist $n$ distinct hyperplanes $G_{i1},...,G_{in}$ such that:
    \begin{enumerate}
        \item $x_i$ does not lie on any of these hyperplanes, and
        \item all the other nodes in $\mathcal{J}$ lie on at least one of these hyperplanes.
    \end{enumerate}
\end{definition}
The following result is fundamental for our purposes. We give a short proof for completeness.
\begin{theorem}[Theorem 1 in \cite{chung1977}]\label{thm:uni-inter}
Let $\mathcal{J}:=\{x_1,...,x_N\}$ be a lattice  of $N=p_n$ distinct nodes of $\RR^d$. If $\mathcal{J}$ satisfies the GC condition as in Definition~\ref{def:gc}, then the polynomial interpolation problem with respect to ${\color{black}\Theta_{\mathcal{J}}}$ is poised for ${\color{black}\mathcal{P}_n}$.
\end{theorem}
\begin{proof}
    Let $Y=\{y_i\}_{i=1,...,N}$ be an arbitrary vector of values. It can be directly verified that the following polynomial of total degree $\leq n$ interpolates $Y$ on $\mathcal{J}$:
    \[
        P(x) = \sum_{i=1}^N y_i \prod_{j=1}^n{{G_{ij}(x)}\over {G_{ij}(x_i)}},
    \]
    {\color{black}where \(G_{i1},\dots,G_{in}\) are distinct hyperplanes associated with \(x_i\in\mathcal J\), and they satisfy the conditions stated in Definition \ref{def:gc}.} Since $Y$ was arbitrary, uniqueness follows.
\end{proof}

Next, we recall the definition of $n$-th order principal lattice \cite{nicolaides1972}.
\begin{definition}
Let $I_n = \{0, {1\over n}, {2\over n}, ..., 1\}$. Further, denote by $X_i := e_i$ the simplex vertices with $e_i$ being the $i$-th element in standard basis of $\RR^d$, and furthermore put $X_0 := \mathbf{0}\in\RR^d$. Then the $n$-th order principal lattice (corresponding to the coordinate simplex) is defined as
\[
    B(n,d) := \biggl\{x \in \RR^d \Big| x = \sum_{i=0}^d \gamma_i X_i,\; \sum_{i=0}^d \gamma_i =1, \;\gamma_i \in I_n \biggr\}.
\]
Note that $B(n,d)$ contains precisely $p_n={n+d \choose n}$ points.
\end{definition}

It was proven in \cite{chung1977} Theorem 4, that $B(n,d)$ satisfies the GC condition which, by Theorem~\ref{thm:uni-inter}, guarantees poisedness \cite{GascaMariano2000}. This is almost what we want, and indeed as we will show, the uniform grid contains a scaled principal lattice.

\begin{lemma}\label{lem:rescaled-lattice}
    Let $J_{n}:=\{0,1,...,n\}$, and put
    \begin{equation}\label{def:A}
        A(n,d) := n \cdot B(n,d) = \biggl\{ x \in \RR^d \Big| x = \sum_{i=0}^d \gamma_i X_i, \sum_{i=0}^d \gamma_i =n, \gamma_i \in J_n \biggr\} \subset \Gg_n.
    \end{equation}
  
    Then the set $A(n,d)$ satisfies the GC condition of degree $n$.
\end{lemma}
\begin{proof}
    The proof is very similar to the one in \cite{chung1977} for $B(n,d)$, and we provide it for completeness.

    Let $\{\gamma_k:\RR^d\to\RR\}_{k=0,...,d}$ be the barycentric coordinate functions associated with the simplex $X_0,...,X_d$. Define the following hyperplanes $H_{km}:=\{x | \gamma_k(x) = m\}$, $k=0,...,d$, $m=0,...,n$. These hyperplanes contain all the nodes in $A(n,d)$.

    Now let $x\in A(n,d)$ be a lattice node, and let $\beta_k := \gamma_k(x)$. For each $k=0,..,d$ with $\beta_k > 0$, the union of the hyperplanes $\{H_{kj}\}_{k=0,...,d}^{j=0,...,\beta_k -1}$ doesn't contain the node $x$. Since $\sum_{k=0}^d\beta_k =n$ we have exactly $n$ hyperplanes in the set $\{H_{kj}\}_{k=0,...,d}^{j=0,...,\beta_k -1}$.

    We still have to show why the union of $\{H_{kj}\}_{k=0,...,d}^{j=0,...,\beta_k -1}$ contains all the other nodes. Let $y \neq x \in A(n,d)$. Again, we have $\sum_{k=0}^d \gamma_k(y)=n$, therefore there is at least one $k \in \{0,...,d\}$ such that $\gamma_k(y) < \gamma_k(x)=\beta_k$. \end{proof}

To clarify the proof, we give an example for $d=2, n=2$. For the node $x=(0,1)$ we have $\gamma_0(x)=1, \gamma_1(x)=0,\gamma_2(x)=1$ because $x = 1\cdot (0,0) + 0\cdot (1,0) +1\cdot (0,1)$. Following the proof, $H_{00}$ and $H_{20}$ are the hyperplanes which contain all the nodes except for $x$ (see figure \ref{fig:H}).
\begin{figure}
\begin{center}
    \includegraphics[width=0.4\textwidth]{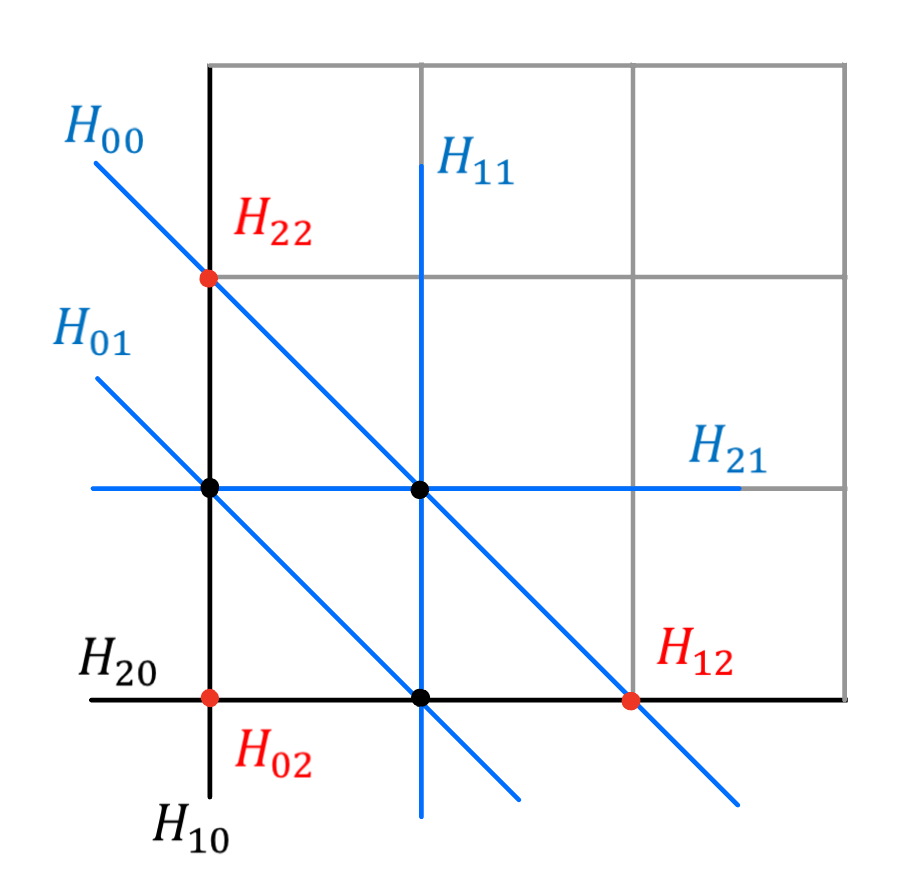}
\end{center}
\caption{Demonstration of the proof of Lemma~\ref{lem:rescaled-lattice}.}
\label{fig:H}
\end{figure}

\begin{remark} For the scaled principal lattice $A(n,d)$, the interpolation polynomial $P(x)$ of degree $\leq n$ for any $\{y_a\}_{a \in A(n,d)}$ is given by:
\[
    P(x)=\sum_{a \in A(n,d)}y_a\prod_{k=0}^d\prod_{m=0}^{\gamma_k(a)-1}{{\gamma_k(x)-m}\over {\gamma_k(a)-m}}.
\]
\end{remark}

{\color{black}
For any $t\in \RR^d$ and a set $A\subseteq \mathbb{R}^d$, we denote by $A_t$ the shifted set $A_t := \{a+t: a\in A\}$.
The following result is immediate.
    \begin{proposition}\label{prop:shifted-GC-set}
        If $A$  satisfies the GC condition of degree \(n\), then $A_t$ satisfies the GC condition of degree $n$ for any $t\in\mathbb{R}^d$.
    \end{proposition}

}

\begin{comment}
    {\color{black}\begin{corollary}\label{cor:shifted-GC-set}
    For any $t\in \RR^d$ and a set $A$ that satisfies the GC condition of degree \(n\), define the shifted set $A_t := A + t$. Then $A_t$ also satisfies the GC condition of degree $n$.
\end{corollary}
\todo[inline]{This is a very very straightforward result, I suggest instead to add a remark after Definition~\ref{def:gc} that GC is preserved under uniform shifts.}
\begin{proof}
    Let \(A=\{x_1,\dots,x_M\}\) satisfy the GC condition of degree \(n\) with hyperplanes \(G_{i1},\dots,G_{in}\) associated with \(x_i\). For a shift vector \(t\in\RR^d\), set \(G'_{ik}:=G_{ik}+t\). Clearly, \(G'_{i1},\dots,G'_{in}\) are distinct hyperplanes. Since \(x_i\notin G_{ik}\) we have \(x_i+t\notin G'_{ik}\), and if \(x_j\in G_{ik}\) for \(j\ne i\) then \(x_j+t\in G'_{ik}\). Therefore the shifted set \(A_t=\{x_1+t,\dots,x_M+t\}\) satisfies the GC condition of degree \(n\).
\end{proof}}
\end{comment}

We are finally in a position to state and prove the main result of this section.

\begin{theorem}\label{thm:grid-is-good}
    Let $\ci \subset \RR^d$ be a finite set of $n$ distinct points in $\RR^d$, and let $m=\mu(\ci)$ be its discrete moment order. Let $\ell$ be the unique integer such that $p_{\color{black}2\ell-2}<n\leq p_{\color{black}2\ell-1}$. Then for every $N\geq \ell$ we have
    \begin{equation*}
        \rank(V_{\leq m}(\Gg_N)) \geq n,
    \end{equation*}
    so the conclusion of Lemma~\ref{lem:exact-scaling} holds for $\sss=\Gg_N$.
\end{theorem}
%\todo[inline]{Below, please use the terminology consistently. The polynomial interpolation problem with respect to $\Theta_{...}$ is poised for $P_{...}$, but sometimes you use $A_s(2\ell-1,d)$ directly as the set of functionals {\color{blue}the second terminology can also be used as in Theorem \ref{thm:uni-inter}, I changed it to match the terminology in \ref{thm:poised}. }}
\begin{proof}
    {\color{black}Let $s= -(\ell-1,\dots,\ell-1)^T \in \RR^d$ and define $A_s(2\ell -1, d) := A(2\ell -1,d) + s$. By Lemma \ref{lem:rescaled-lattice} and Proposition \ref{prop:shifted-GC-set} we have that $A_s(2\ell -1, d)$ satisfies the GC condition of degree $2\ell -1$.}
    Using  {\color{black}Theorem \ref{thm:uni-inter} {\color{black}(or alternatively Corollary \ref{cor:l-poised})} we get that the polynomial interpolation problem with respect to $\Theta_{A_s(2\ell -1, d)}$ is poised for $\text{span}\{P_{2\ell -1}\}$, thus by Theorem \ref{thm:poised} we get}
    \begin{align*}
        \big|V_{\leq {\color{black}2\ell-1}}(A_{\color{black}s}({\color{black}2\ell-1},d))\big|  \neq 0 & \implies \rank(V_{\leq {\color{black}2\ell-1}}(A_{\color{black}s}({\color{black}2\ell-1},d))) = {{{\color{black}2\ell-1}+d} \choose d} = p_{\color{black}2\ell-1} \geq n\\
        A_{\color{black}s}({\color{black}2\ell-1},d) \subset \Gg_{\ell} &\implies \rank(V_{\leq {\color{black}2\ell-1}}(\Gg_{\ell})) \geq {{{\color{black}2\ell-1}+d} \choose d} \geq n.
    \end{align*}
    Now notice that since $\rank V_{\leq m}(\ci)=n$ we must have $p_m \geq n$ and therefore $m \geq {\color{black}2\ell-1}$. This implies
    $$
        \rank(V_{\leq m}(\Gg_{N})) \geq \rank(V_{\leq m}(\Gg_{\ell})) \geq \rank(V_{\leq {\color{black}2\ell-1}}(\Gg_{\ell})) \geq n,
    $$
    finishing the proof.
\end{proof}

{\color{black}\begin{remark}
    Theorem \ref{thm:grid-is-good} can be reformulated without explicitly referring to the set $\mathcal{\ci}$ and its discrete moment order $m = \mu(\mathcal{\ci})$ by instead imposing the  assumption that $p_m \ge n$ for given $m,n$.
\end{remark}}

{\color{black}We now introduce an alternative condition for poisedness, following the work \cite{dyn2014multivariate}.}

{\color{black}\begin{definition}[\cite{dyn2014multivariate}]\label{def:lower-set}
    A finite set $L \subset \mathbb{N}_{0}^d$ is a \textit{lower set} if whenever $\mu \in L$ and $0 \leq \alpha \leq \mu$ then $\alpha \in L$, where $\alpha \leq \beta$ means that $\alpha_j \leq \beta_j$ for $j=1,\dots,d$.
\end{definition}

\begin{theorem}[Theorem 1 in \cite{dyn2014multivariate}]\label{thm:l-poised}
    Let $L$ be a lower set and $F_L := \{x_{\alpha} : \alpha\in L\}$. Then the polynomial interpolation problem with respect to $\Theta_{F_L}$ is poised for $\mathcal{P}_L:=\text{span}\{x^{\alpha} : \alpha \in L\}$.
\end{theorem}}
%\todo[inline]{Please double check. 1. Why use $span\{P_{2\ell-1}\}$ and not just $P_{2\ell-1}$? {\color{blue}because $P_{2\ell-1}$ is a basis and not the vector space and according to the terminology of Theorem \ref{thm:poised}, a set of functionals is poised for a linear subspace $\mathcal{P}$...} 2. You need to explain why $\mathcal{P}_{A(2\ell-1,d)}\equiv P_{2\ell-1}$. Also in your proof you take the shifted $A_s$ but for the polynomial space this index set should be non-negative... {\color{blue}done and fixed}}
{\color{black}
    \begin{corollary}\label{cor:l-poised}
        The polynomial interpolation problem with respect to $\Theta_{A_s(2\ell -1, d)}$ is poised for $\mathcal{P}_{2\ell-1}:=\text{span}\{P_{2\ell -1}\}$, where $P_{2\ell -1}$ is the standard monomial basis of polynomials of total degree $\leq 2\ell -1$ and $s := -(\ell-1,\dots,\ell-1)^T$.
    \end{corollary}
    \begin{proof}
        It is easy to verify that the set $A(2\ell -1, d)$ is a lower set as in Definition~\ref{def:lower-set}. Now we show that $\mathcal{P}_{A(2\ell-1,d)}\equiv \mathcal{P}_{2\ell -1}$. By the definition of $A(2\ell-1,d)$ in \eqref{def:A}, any $\alpha \in A(2\ell-1,d)$ has the form $(\alpha_1,\dots,\alpha_d)$ where $\alpha_i \in \{0,1,\dots,2\ell -1\}$ and $\sum_{i=0}^d\alpha_i = 2\ell -1$ which is equivalent to $\sum_{i=1}^d\alpha_i \leq 2\ell-1$. Hence the set $\{x^{\alpha} : \alpha \in A(2\ell-1,d)\}$ is precisely the set of monomials of total degree $\leq d$.
        To complete the proof, let $F_{A(2\ell-1,d)} := \{\alpha + s : \alpha \in A(2\ell-1,d)\}$. It is easy to see that $F_{A(2\ell-1,d)} \equiv A_s(2\ell-1,d)$, therefore, using Theorem \ref{thm:l-poised}, we get the desired result.
        %An alternative way to prove that the polynomial interpolation problem with respect to $\Theta_{A_s(2\ell -1, d)}$ is poised for $\mathcal{P}:=\text{span}\{P_{2\ell -1}\}$ is to use the result on \textit{lower sets} \cite{dyn2014multivariate}. Since $A_s(2\ell -1, d)$ is a lower set, by \cite[Theorem 1]{dyn2014multivariate}, there exists a unique polynomial $f \in \mathcal{P}$ that interpolates any $Y \in \RR^{\#\Theta}$ on $A_s(2\ell -1, d)$.
        %there exists a unique $f \in \mathcal{P}$ such that $\Theta f = Y$
    \end{proof}
        %\todo[inline]{This gives an alternative proof of Theorem~\ref{thm:grid-is-good}..?}
Note that corollary \ref{cor:l-poised}, gives an alternative proof of the poisdness result in Theorem \ref{thm:grid-is-good}.
}

\section{Vandermonde matrices in the super-resolution regime}\label{sec:vandermonde}

In this section we come back to the super-resolution problem, and consider the Vandermonde matrices $U(x_1,\dots,x_n;\Gg_N)$ in \eqref{eq:Van} in the single cluster setting. We prove the multidimensional analogue of \cite[Theorem 2.3]{batenkov2021}, with the caveat that the geometry of the nodes stays fixed as $\srf\gg 1$. We start by defining the cluster geometry. Naturally, since the sampling set is an integer grid, we must restrict the nodes to $\TT^d:=[-\pi,\pi)^d \equiv \left(\RR \mod 2\pi\right)^d$ to avoid aliasing.

\begin{definition}
    For $x,y \in \TT^d$, we denote the wrap-around distance by
    \[ \|x-y\|_{\TT^d}:=\min_{r \in (2\pi \ZZ)^d}\|x-y+r\|_{\infty}. \]
\end{definition}

\begin{definition}
    Let $\mathcal{X}=\{x_1,...,x_n\} \subset \TT^d$. If for some $\tau>1$ and $0<\Delta<\pi/\tau$ we have
    \[ \forall x,y \in \mathcal{X}, x \neq y: \quad \Delta \leq \|x-y\|_{\TT^d} \leq \tau \Delta,\]
    then $\mathcal{X}$ is said to form an $(\Delta,\tau,n)$-cluster.
\end{definition}

In what follows we fix a node set $\di=\{y_1,\dots,y_n\} \subset \td$ such that the (non-wrapped) minimal distance between any two nodes is
$$
\rho = \rho(\di) := \min_{i\neq j} \|y_i-y_j\|_{\infty}.
$$

Now let $\ci=\Delta\di$ with $\Delta < \frac{1}{2\pi}$, so that $\ci$ forms an $(\Delta', \tau', n)$-cluster for $\Delta'=\rho\Delta$ and some $\tau'\leq \frac{2\pi}{\rho}$. (Otherwise, there exist $y,y'\in\di$ such that $\|y-y'\|_{\TT^d} \geq \frac{2\pi}{\rho}\Delta' = 2\pi\Delta$, a contradiction).

Further, let $\Gg_N = \{0,\pm 1,..., \pm N\}^d$ be the uniform symmetric grid of side length $2N+1$.

Let $U:=U(x_1,\dots,x_n;\Gg_N)$. Denote $\epsilon:=N\Delta$. We have
$$
U^* U = D(\ci,\Gg_N) = D\Bigl(\epsilon \frac{\di}{N},\Gg_N\Bigr) = D_{\epsilon}\Bigl(\frac{\di}{N},\Gg_N\Bigr). 
$$

We further define the rescaled Dirichlet kernel
\begin{align*}
\mathcal{K}(x,y)&=\mathcal{K}_N(x,y):=\frac{1}{(2N)^d} \mathcal{D}_{\Gg_N}(x/N,y/N) \\
\implies K_{\epsilon}(\di) &= \frac{1}{(2N)^d} D_{\epsilon}\Bigl(\frac{\di}{N},\Gg_N\Bigr) = D_N(\ci).
\end{align*}

We have the following result.

\begin{theorem}\label{thm:sr}
    % For $\ci$ forming a $(h,\tau,n)$-cluster, let $D_N := \frac{1}{N^d}D_{\epsilon}(\frac{\mathcal{Y}}{N},\Gg_N) = \frac{1}{N^d}\big[\mathcal{D}_{\Gg_N}(\epsilon\frac{y_i-y_j}{N})\big]_{1\leq i,j \leq s}$, where $\epsilon:=Nh$ and $D_{\epsilon}(\frac{\mathcal{Y}}{N},\Gg_N) = D_{h}(\ci,\Gg_N)$.
    With the above notations, let $m=\mu(\di)=\mu(\ci)$ be the discrete moment order of $\di$.
    
    % {\color{gray}\sout{
    % There exists $\epsilon_0>0$ such that for all $N,\Delta$ satisfying $N\Delta<\epsilon_0$,}}
    For every $\epsilon_0>0,\alpha\in(0,1)$ there exists $N_0$ such that for all $N\geq N_0$ and $\alpha\epsilon_0 < N\Delta < \epsilon_0$,    
    the eigenvalues of $D_{N}$ split into $m+1$ groups
\[
\lambda_{0,0}(D_N) \asymp {(N\Delta)}^{0}, \quad \{\lambda_{1,j}(D_N)\}_{j=1}^{t_1}\asymp{(N\Delta)}^2, ... ,\{\lambda_{m,j}(D_N)\}_{j=1}^{t_m}\asymp{(N\Delta)}^{2m},
\]
where $t_k:=\rank(V_{\leq k}(\ci))-\rank(V_{\leq k-1}(\ci))$, $k=1,\dots,m$. {\color{black}Here $\asymp$ means 'exact scaling' as defined in Corollary \ref{cor:tightness}.}                     
\end{theorem}
\begin{proof}
    With the identification $D_N(\ci) = K_{\epsilon}(\di)$ as above, we would like to apply Theorem~\ref{Thm:main} and Corollary~\ref{cor:tightness} to the kernel $\mathcal{K}=\mathcal{K}_N$. However, $\mathcal{K}$ depends on $N$, so we need to make sure that all the estimates in the proofs are uniform in $N$.

        \begin{definition}[sinc kernel]
            Let $\sinc(x):= \begin{cases} 
                                \frac{\sin(x)}{x}, & \text{if } x \neq 0, \\
                                 1, & \text{if } x = 0.
                            \end{cases}$. Given $\ci \in \RR^d$, the corresponding sinc kernel matrix is defined as follows:
            \begin{equation}
                \Sinc(\ci) := \biggl[ \prod_{i=1}^d\sinc(x_i-x_i')\biggr]_{x,x' \in \ci}
            \end{equation}
            For $\epsilon > 0$, the flat limit version of the $\sinc$ kernel matrix is:
            \begin{equation}
                \Sinc_{\epsilon}(\ci) := \biggl[ \prod_{i=1}^d\sinc(\epsilon x_i-\epsilon x_i')\biggr]_{x,x' \in \ci}
            \end{equation}
        \end{definition}
        Outline of the proof: we aim to prove that the eigenvalues of $D_N$ converge to $\Sinc$ kernel when $N\to\infty$. The scaling of the eigenvalues of $\Sinc$ can be obtained using Corollary \ref{cor:tightness} (the constants $\tilde{\lambda}_{k,j}$ don't depend on $N$). First we show element-wise convergence of $D_N$ to $\sinc$ and equicontinuity of $D_N$ in $\epsilon < \epsilon_0$ to show uniform convergence of $D_N$ to $\sinc$ (uniform in $\epsilon < \epsilon_0$ and independent of $N$). Then using the Bauer–Fike Theorem, we prove convergence of the eigenvalues.
        \par We start with some definitions.
        \begin{theorem}[The Bauer–Fike Theorem]\label{thm:b-f}
            For a diagonalizable matrix $A$ let $E$ be it's eigenvectors matrix, $A=E\Lambda {E}^{-1}$ where $\Lambda$ is a diagonal matrix. Let $\mu$ be an eigenvalue of $A_N$, then there exists $\lambda \in \Lambda(A)$ such that 
            \begin{equation}
                |\lambda - \mu| \leq \kappa_p(E)\|A-A_N\|_p
            \end{equation}
            where $\kappa_p(E) := \|E\|_p\|{E}^{-1}\|_p$ is the condition number of $E$ associated with the norm $\|\cdot\|_p$.
        \end{theorem}
        \begin{definition}[Element-wise convergence]
           Let \(\{A_N\}_{N=1}^{\infty}\), where \(A_N \in \mathbb{R}^{s \times s}\), be a sequence of matrices, and let \(A \in \mathbb{R}^{s \times s}\) be a matrix. The sequence \(\{A_N\}_{N=1}^{\infty}\) is said to converge element-wise to \(A\) if the following condition is satisfied:

            For every \(\delta > 0\), there exists a positive integer \(M(\delta)\) such that for all \(N \geq M(\delta)\), the inequality
            \[
                \big|(A_N)_{i,j} - (A)_{i,j}\big| \leq \delta
            \]
            holds for all \(1 \leq i, j \leq s\).

        \end{definition}
        \begin{definition}[Equicontinuity]\label{def:equic}
            Let $\{f_n(x)\}$ be a sequence of functions. $\{f_n(x)\}$ is equicontinuous if for every $\delta > 0$, there exists $\nu > 0$ (independent of $n$) such that for all $n$ and all $x, y \in [a, b]$ with $|x - y| < \nu$ we have $|f_n(x) - f_n(y)| < \delta$.
        \end{definition}
        \begin{definition}[Uniform convergence]
            A sequence of functions $\{ f_n(x) \}$ converges uniformly to $f(x)$ on $S$ if for every $\epsilon > 0$, there exists an integer $N$ such that for all $n \geq N$ and for all $x \in S$, $| f_n(x) - f(x) | < \epsilon$.
        \end{definition}
        \begin{theorem}\label{thm:uni-con}
            Let $C(J)$ be the space of continuous functions on a closed interval $J \subset \RR$. A sequence in $C(J)$ is uniformly convergent if and only if it is equicontinuous and converges pointwise to a function.
        \end{theorem}
        \begin{proof}
            The first direction is an immediate corollary of  {\color{black}\cite[Theorem 7.25]{rudin1964}}. Since the sequence converges pointwise, we can replace the subsequence in the proof of (b) by the sequence itself. The second direction is Theorem 7.24.
        \end{proof}
        \begin{lemma}\label{lem:dir-con}\label{lem:poin}
            For every $\epsilon > 0$ and $\ci \in \RR^d$, the Dirichlet kernel matrix $D_N(\ci)$ converges element-wise to the $\sinc$ kernel matrix $\Sinc_{\epsilon}(\mathcal{Y})$ when $N \to \infty$.
        \end{lemma}
        \begin{proof}
            \begin{equation}
                \big(D_N(\ci)\big)_{i,j} = \frac{1}{(2N)^d} \big(D_{\epsilon}\Bigl(\frac{\di}{N},\Gg_N\Bigr)\big)_{i,j} = \frac{1}{(2N)^d}\mathcal{D}_{\Gg_N}(\epsilon({y_i \over N} - {y_j \over N})) = \frac{1}{(2N)^d}\sum_{\omega \in \Gg_N}\exp\left( \imath \langle \epsilon({y_i \over N} - {y_j \over N}),\omega \rangle \right)
            \end{equation}
            Denote by $y:=y_i - y_j$ and $y:=(y^{(1)},\dots,y^{(d)})$. First, we consider the case where $d=1$:
            \begin{align*}
                &\frac{1}{2N}\sum_{k=-N}^N \exp(\imath \epsilon k{y \over N}) = \frac{1}{2N}(1+2\sum_{k=1}^N \cos{(\epsilon k{y \over N})}) = \frac{1}{2N}\frac{\sin((N+{1\over 2}){\epsilon y \over N})}{\sin({\epsilon y \over 2N})} \\
                & \to_{N\to \infty} \frac{1}{2N}\frac{\sin((1+{1\over 2N})\epsilon y)}{{\epsilon y \over 2N}} \to_{N\to \infty} {\sin(\epsilon y) \over \epsilon y}= \sinc(\epsilon y).
            \end{align*}
            For $d > 1$, using the above, we get
            \begin{align*}
                \frac{1}{(2N)^d}\sum_{\omega \in \Gg_N}\exp\left( \imath \langle \epsilon({y \over N}),\omega \rangle \right) = \prod_{i=1}^d \frac{1}{2N}\sum_{k=-N}^N \exp(\imath \epsilon k{y^{(i)} \over N}) \to_{N\to \infty} \prod_{i=1}^d \sinc(\epsilon y^{(i)}).
            \end{align*}
        \end{proof}

            We proved that each entry of $D_N(\ci)$ converges point-wise to each entry of $\Sinc_{\epsilon}(\mathcal{Y})$. Now we show that this convergence is uniform for $\epsilon \in [0, \epsilon_0]$ for a given $\epsilon_0$. We show that using Theorem \ref{thm:uni-con}. First we need to show that the sequence $\{\big(D_N(\ci)\big)_{i,j}\}$ is equicontinuous.
            \par For fixed $y \in \mathcal{Y}$, let $f_N(\epsilon) := \frac{1}{2N}\sum_{k=-N}^N \exp(\imath \epsilon k{y \over N})$ and $f(\epsilon):=\sinc(\epsilon y)$. We prove the following result.
            \begin{lemma}\label{lem:equi}
                Let  $\epsilon_0 > 0$ be given. For every $\delta > 0$ there exists $\nu > 0$ (independent of $N$) such that for all $N$ and all $\epsilon, \epsilon ' \in [0, \epsilon_0]$ with $|\epsilon - \epsilon '| < \nu$ we have $|f_n(\epsilon) - f_n(\epsilon ')| < \delta$.
            \end{lemma}
            \begin{proof}
                Let $\delta > 0$. we want to find $\nu(\delta) > 0$ such that for all $N$ and $\epsilon, \epsilon ' \in [0, \epsilon_0]$ with $|\epsilon - \epsilon '| < \nu$ we have that $|f_N(\epsilon) - f_N(\epsilon ')| < \delta$. Using that $|1-e^{ix}| = 2|\sin({x \over 2})|$ and $\sin(x) \leq x$ for $x \in [0,{\pi \over 2}]$ we get
                \begin{align*}
                    |f_N(\epsilon) - f_N(\epsilon ')| &= |\frac{1}{2N}\sum_{k=-N}^N \exp(\imath \epsilon k{y \over N}) - \frac{1}{2N}\sum_{k=-N}^N \exp(\imath \epsilon ' k{y \over N})| \\
                    &= | \frac{1}{2N}\sum_{k=-N}^N \exp(\imath \epsilon k{y \over N})(1 - \exp(\imath (\epsilon - \epsilon ') k{y \over N})) | \\
                    &\leq \frac{1}{2N}\sum_{k=-N}^N|1 - \exp(\imath (\epsilon - \epsilon ') k{y \over N})| \\
                    &= \frac{1}{2N}\sum_{k=-N}^N 2|\sin((\epsilon - \epsilon ') k{y \over 2N})| \\
                    &\leq \frac{2}{2N}\sum_{k=0}^N 2k|\epsilon - \epsilon '|{y \over 2N} = {1 \over N^2}{N+N^2 \over 2}\nu y \\
                    &= ({1\over 2N} +{1\over 2})y\nu \leq y\nu.
                \end{align*}
                Thus we choose $\nu := {\delta \over y}$. The same holds for $f_N(\epsilon) : =\frac{1}{(2N)^d}\sum_{\omega \in \Gg_N}\exp\left( \imath \langle \epsilon({y \over N}),\omega \rangle \right)$.
            \end{proof}
            % \todo[inline]{Need to be careful here. The limiting behaviour $N\to\infty$ could depend in principle on $\epsilon$, i.e. given $\epsilon>0$ there exists $M(\epsilon)$ s.t. $\forall N>M(\epsilon)$... I think the convergence should be uniform in $\epsilon<\epsilon_0$. So in the end we would say something like: ``for all $\srf$ larger than some constant and $N$ larger than some other constant, the estimates hold.''}
        %\begin{lemma}\label{lem:dir-bound}
        %    For $\epsilon > 0$ and $\ci \in \RR^d$, the Dirichlet kernel matrix $D_N(\ci)$ is uniformly bounded in $2$-norm.
        %\end{lemma}
        %\begin{proof}
        %    \begin{align*}
        %        \bigg\|\frac{1}{(2N)^d}\sum_{\omega \in \Gg_N}\exp\left( \imath \langle \epsilon({y \over N}),\omega \rangle \right)\bigg\|_{2} = \prod_{i=1}^d \bigg\|\frac{1}{2N}\sum_{k=-N}^N \exp(\imath \epsilon k{y_i \over N})\bigg\|_{2} \leq \prod_{i=1}^d \bigg\|\frac{1}{2N} \cdot 2N\bigg\|_2 = 1.
        %    \end{align*}
        %\end{proof}
        Using Lemma \ref{lem:poin} and \ref{lem:equi} together with Theorem \ref{thm:uni-con} we get the following Corollary.
        \begin{corollary}\label{cor:uni-conv}
            Let $\epsilon_0$ be given. For every $\delta > 0$ there exists $M=M(\epsilon_0,\delta)$ such that for all $N > M$
            \begin{equation}
                \big|\big(D_N(\ci)\big)_{i,j} - \big(\Sinc_{\epsilon}(\mathcal{Y})\big)_{i,j}\big| \leq \delta,\quad \forall \epsilon \in [0, \epsilon_0]. 
            \end{equation}
            
            \begin{proof}
                We take $\nu$ in Definition~\ref{def:equic} to be the minimum of all the $\nu_{i,j}^{(\ell)}$ according to the proof of Lemma~\ref{lem:equi}, with $\ell=1,\dots,d$. Note that for $y_{i,j}^{(\ell)}=0$ we have $\nu_{i,j}^{(\ell)}=\infty$.
            \end{proof}
            
        \end{corollary}
        \begin{corollary}[Uniform convergence of eigenvalues]
        
        Let $\epsilon_0>0$ be given. For every $\delta '>0$ there exists $M=M(\epsilon_0,\delta ')$ s.t. for all $N>M$ there is an eigenvalue $\mu_k^{(N)}$ of $D_N(\ci)$ with
                    \begin{equation}\label{eq:eig-pert-sinc-dirichlet}
                |\lambda^{(\epsilon)}_k - \mu_k^{(N)}| \leq \delta ', \quad \forall \epsilon\in[0,\epsilon_0]
            \end{equation}
            where $\lambda^{(\epsilon)}_k$ is an eigenvalue of $\Sinc_{\epsilon}(\mathcal{Y})$.
        
            % Let $\mu$ be an eigenvalue of $D_N(\ci)$. For given $\epsilon > 0$, there exists an eigenvalue $\lambda$ of $Sinc_{\epsilon}(\mathcal{Y})$ such that for every $\delta > 0$ there exists $M(\epsilon)$ such that for all $N > M(\epsilon)$
            % \begin{equation}\label{eq:eig-pert-sinc-dirichlet}
            %     |\lambda - \mu| \leq \delta
            % \end{equation}
        \end{corollary}
        \begin{proof}
            Since $\Sinc_{\epsilon}(\mathcal{Y})$ is a symmetric matrix, the eigenvectors matrix $E$ is orthogonal in the eigenvalue decomposition of $\Sinc_{\epsilon}(\mathcal{Y})$ i.e $\Sinc_{\epsilon}(\mathcal{Y})=E\Lambda {E}^{-1}$. Thus we have $\kappa_2(E) = 1$. Using the Bauer–Fike theorem \ref{thm:b-f} with Corollary \ref{cor:uni-conv} and $\delta := {\delta ' \over n\sqrt{n}}$, we get:
            \begin{align*}
                &|\lambda - \mu| \leq \|\Sinc_{\epsilon}(\mathcal{Y}) - D_N(\ci)\|_2 \\
                &\leq \sqrt{n}\|\Sinc_{\epsilon}(\mathcal{Y}) - D_N(\ci)\|_{\infty} \\
                &= \sqrt{n}\max_i\sum_{j=1}^n\big|(\Sinc_{\epsilon}(\mathcal{Y}))_{i,j} - (D_N(\ci))_{i,j}\big| \\
                &\leq \sqrt{n}n\delta = \delta '.
            \end{align*}
        \end{proof}
        To complete the proof of Theorem~\ref{thm:sr}, we choose $\delta'=O\biggl((\alpha \epsilon_0)^{2m}\biggr)$  in \eqref{eq:eig-pert-sinc-dirichlet} and apply Corollary~\ref{cor:tightness}.
        The only remaining question is to show that $\det(RWR) \neq 0$. We will show that the limit $N\to\infty$ the rank of $W=\lim_{N\to\infty} \frac{1}{(2N)^d} W_N$ is full, and as a result, $\det\left(\tilde{R}W\tilde{R}^T\right)$ is uniformly bounded away from zero. By definition of the Wronskian, we write
        \begin{align*}
        \frac{1}{(2N)^d}W_N &= \frac{1}{(2N)^d}\bigg[{{\mathcal{D}_{\Gg_N}^{(\alpha,\beta)}({x \over N},{y \over N})|_{x,y=0}}\over{\alpha! \beta!}}\bigg]_{\alpha,\beta \in \PP_m} = \bigg[\frac{1}{(2N)^d}\sum_{\omega \in \Gg_N}{1\over N^{|\alpha|+|\beta|}}{{(i\omega)^{\alpha}(-i\omega)^{\beta}}\over{\alpha! \beta!}}\bigg]_{\alpha,\beta \in \PP_m} \\
        &= \bigg[\frac{1}{(2N)^d}\sum_{\omega =(\omega_1,\dots,\omega_d)\in \Gg_N}{{(i)^{|\alpha|}(-i)^{|\beta|}}\over{\alpha! \beta!}}\big({\omega_1 \over N}\big)^{\alpha_1+\beta_1}\cdot ... \cdot \big({\omega_d \over N}\big)^{\alpha_d+\beta_d}\bigg]_{\alpha,\beta \in \PP_m}\\
        &= \bigg[ \prod_{j=1}^d \sum_{k=-N}^N \frac{1}{2N}\frac{(i)^{\alpha_j}(-i)^{\beta_j}}{\alpha_j! \beta_j!}\big({k\over N}\big)^{\alpha_j+\beta_j}\bigg]_{\alpha,\beta \in \PP_m}.
     \end{align*}
     We get that each entry of $\frac{1}{(2N)^d}W_N$ is a Riemann sum, thus as $N \to \infty$ we get a multiple integral:
     \[
        \lim_{N \to \infty} \prod_{j=1}^d \sum_{k=-N}^N \frac{1}{2N}\frac{(i)^{\alpha_j}(-i)^{\beta_j}}{\alpha_j! \beta_j!}\biggl({k\over N}\biggr)^{\alpha_j+\beta_j} = {1 \over 2^d}\int_{[-1,1]^d}{{(i)^{|\alpha|}(-i)^{|\beta|}}\over{\alpha! \beta!}}x^{\alpha+\beta} \,dx
     \]
     Thus we get:
     \[
        W =\lim_{N\to\infty} \frac{1}{(2N)^d}W_N = \bigg[ {1 \over 2^d}\int_{[-1,1]^d}{{(i)^{|\alpha|}(-i)^{|\beta|}}\over{\alpha! \beta!}}x^{\alpha+\beta} \,dx \bigg]_{\alpha,\beta \in \PP_m}.
     \]
        It is immediately seen that $W$ is none other than the Gram matrix for the monomials $\{\boldsymbol{x}^{\boldsymbol{\alpha}}\}_{\boldsymbol{\alpha}\in\PP_m}$:
        \[
            W_{\alpha,\beta} = \langle x^{\alpha}, x^{\beta} \rangle = {1 \over 2^d}\int_{[-1,1]^d}{{(i)^{|\alpha|}(-i)^{|\beta|}}\over{\alpha! \beta!}}x^{\alpha+\beta} \,dx.
        \]
        Since the set of monomials is linearly independent, we have that $W$ is positive definite. Thus $W$ is full rank, and as a result, $\det\left(\tilde{R}W\tilde{R}^T\right) > 0$.
        This completes the proof of Theorem~\ref{thm:sr}.
    \end{proof}

    \begin{remark} \label{rem:weilin}During the final stages of preparation of our manuscript, we became aware that Weilin Li is working on a related topic  \cite{li2024a}. Relative to our results, more explicit conditions and constants are derived for the scaling of the smallest singular value of the multivariate Vandermonde matrix (we consider the entire spectrum); multi-cluster geometry is investigated; only cube and spherical sampling sets are considered; the decay rate in the case of uniform dilations of generic sets is not sharp; finally, a form of geometric characterization condition is used to describe the node geometry (as opposed to the sequence $\{t_k\}$ in our work). \end{remark}

% {\color{black} \begin{remark}
%     This is true also for $\ci \subset \CC^d$.
% \end{remark}}{\color{blue} Not sure because the Taylor expansion in \ref{app:taylor} is for real-valued function, see section 8.4.4 in \cite{zorich} or section 5.5 in \cite{usevich2021}}

           % \begin{definition}[Uniform Convergence]
           % A sequence of matrices $\{A_N\}$ converges uniformly to a matrix $A$ if for every $\delta > 0$ there exists $M$ such that for all $N \geq M$, $\|A_N - A\| \leq \delta$. where $\|\cdot\|$ is a matrix norm.
           %  \end{definition}
           % \begin{definition}[Uniform Boundedness]
           % A sequence of matrices $\{A_N\}$ is uniformly bounded if there exists a constant $C$ such that for all $N$, $\|A_N\| \leq C$.
           % \end{definition}
           %  \begin{theorem}\label{thm:uni-con}
           % If a sequence of matrices $\{A_N\}$ converges element-wise to $A$ and is uniformly bounded then it converges uniformly.
           % \end{theorem}
     
\section{Numerical Experiments}\label{sec:numerics}
    In this section, we describe some numerical example to validate our results and investigate their applicability to the problem of super-resolution.
    
    \subsection{Dirichlet kernel eigenvalue scaling}

    \begin{figure}
        \centering
        \begin{subfigure}[b]{0.45\textwidth}
            \includegraphics[width=0.95\textwidth]{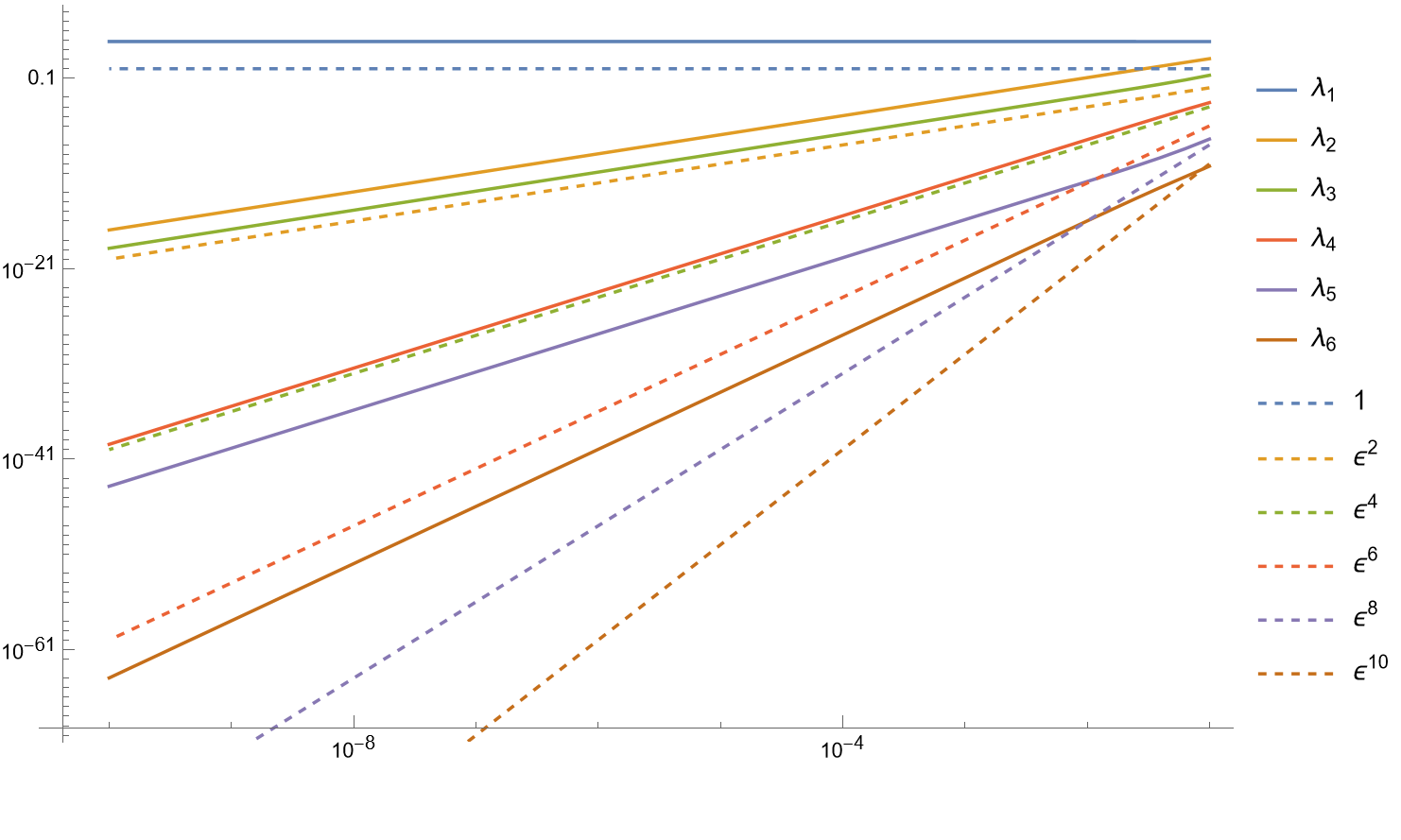}
            \caption{Parabola geometry.}
            \label{fig:alg-var}
        \end{subfigure}
        \begin{subfigure}[b]{0.45\textwidth}
            \includegraphics[width=0.95\textwidth]
        {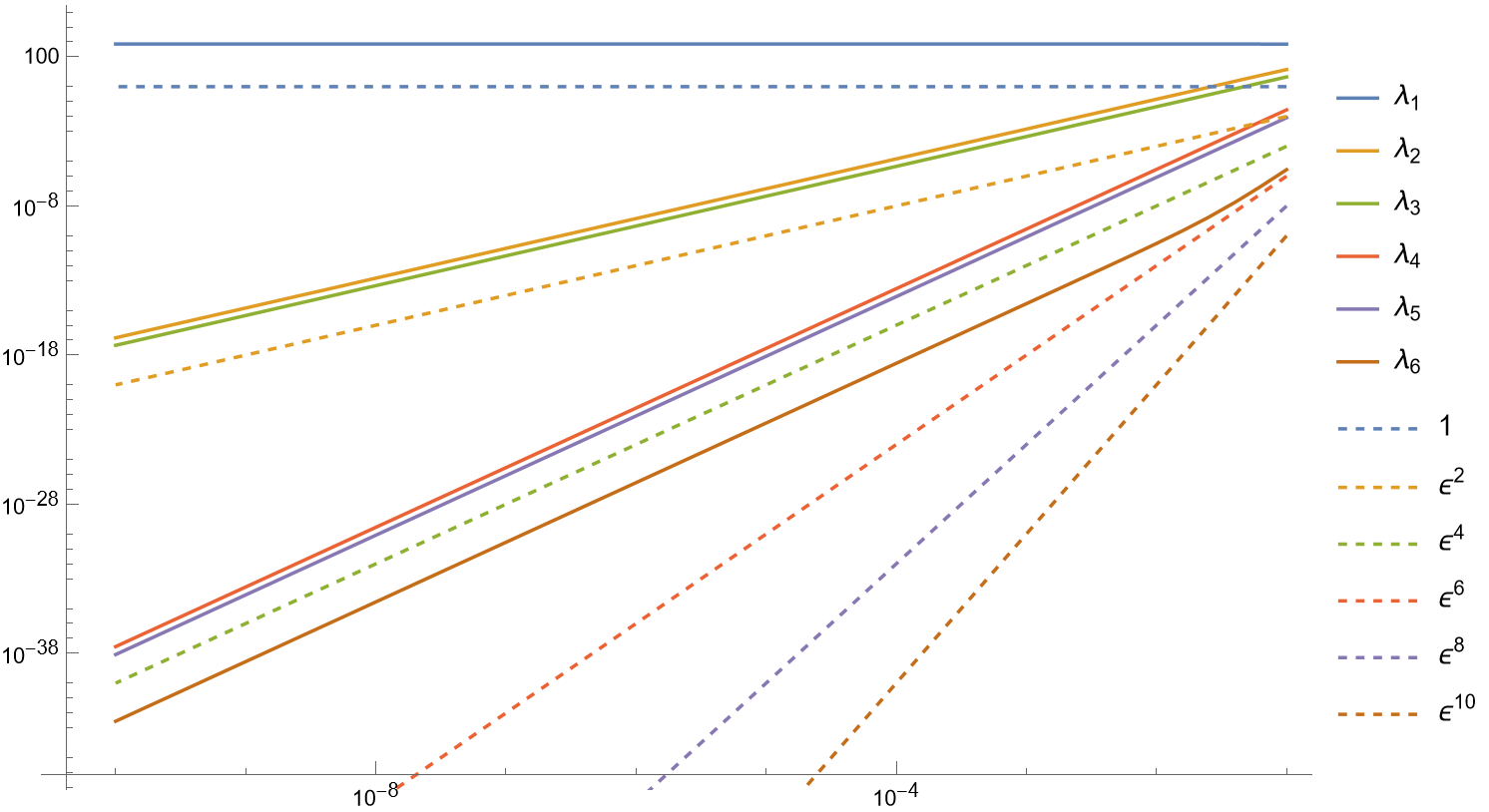}
            \caption{General position geometry. }
            \label{fig:non-alg-var}
        \end{subfigure}
        \caption{Eigenvalues of Dirichlet kernel with samples on the grid for the different node geometries. The $x$-axis is $\epsilon$ denoting the smallest distance between the nodes.}
    \end{figure}

    We start with the numerical study of the eigenvalues 
    of the Dirichlet matrix $D_{\epsilon}(\ci,\Gg_n)$, confirming the scaling predicted  by Lemma \ref{lem:exact-scaling}.
    Let $d=2$, $n=6$ and $\mathcal{X}=\{(x_i,y_i)\}_{i=1}^n$. We checked two scenarios: 
    \begin{enumerate}
        \item All points satisfy $y_i=x_i^2$. We get that $\rank(V_{\leq 3})=6$, where
        \begin{align*} V_{\leq 3} &= \left[\begin{array}{c|cc|ccc|cccc}
                                1 & x_1 & y_1 & x_1^2 & x_1y_1 & y_1^2 & x_1^3 & x_1^2y_1 & x_1y_1^2 & y_1^3 \\
                                1 & x_2 & y_2 & x_2^2 & x_2y_2 & y_2^2 & x_2^3 & x_2^2y_2 & x_2y_2^2 & y_2^3 \\
                                1 & x_3 & y_3 & x_3^2 & x_3y_3 & y_3^2 & x_3^3 & x_3^2y_3 & x_3y_3^2 & y_3^3 \\
                                1 & x_4 & y_4 & x_4^2 & x_4y_4 & y_4^2 & x_4^3 & x_4^2y_4 & x_4y_4^2 & y_4^3 \\
                                1 & x_5 & y_5 & x_5^2 & x_5y_5 & y_5^2 & x_5^3 & x_5^2y_5 & x_5y_5^2 & y_5^3 \\
                                1 & x_6 & y_6 & x_6^2 & x_6y_6 & y_6^2 & x_6^3 & x_5^2y_5 & x_6y_6^2 & y_6^3
                        \end{array}\right] \\ &=
                        \left[\begin{array}{c|cc|ccc|cccc}
                                1 & x_1 & x^2_1 & x_1^2 & x_1^3 & x_1^4 & x_1^3 & x_1^3 & x_1^4 & x_1^6 \\
                                1 & x_2 & x^2_2 & x_2^2 & x_2^3 & x_2^4 & x_2^3 & x_2^3 & x_2^4 & x_2^6 \\
                                1 & x_3 & x^2_3 & x_3^2 & x_3^3 & x_3^4 & x_3^3 & x_3^3 & x_3^4 & x_3^6 \\
                                1 & x_4 & x^2_4 & x_4^2 & x_4^3 & x_4^4 & x_4^3 & x_4^3 & x_4^4 & x_4^6 \\
                                1 & x_5 & x^2_5 & x_5^2 & x_5^3 & x_5^4 & x_5^3 & x_5^3 & x_5^4 & x_5^6 \\
                                1 & x_6 & x^2_6 & x_6^2 & x_6^3 & x_6^4 & x_6^3 & x_5^3 & x_6^4 & x_6^6
                        \end{array}\right] = \begin{bmatrix}
                            V_0 & | & V_1 & | & V_2 & | & V_3
                        \end{bmatrix}
        \end{align*}
        We get that
        \begin{align*}
            t_0&=\rank(V_{\leq 0})=1,\\
            t_1&=\rank(V_{\leq 1})-rank(V_{\leq 0})=3-1=2,\\ t_2&=\rank(V_{\leq 2})-rank(V_{\leq 1})=5-3=2,\\ t_3&=\rank(V_{\leq 3})-\rank(V_{\leq 2})=6-5=1.
        \end{align*}
         The results in Figure~\ref{fig:alg-var} confirm the scaling of the eigenvalues.
        % \begin{figure}
        %     \centering
        %     \includegraphics[width=0.9\textwidth]{singularVandermonde.png}
        %     \caption{Parabola geometry. The x-axis is $\epsilon$ denoting the smallest distance between the nodes}
        %     \label{fig:alg-var}
        % \end{figure}
        \item The set of nodes in $\ci$ are in general positions. We have $\rank(V_{\leq 2})=6$, where
        \begin{align*} V_{\leq 2} = \left[\begin{array}{c|cc|ccc}
                                1 & x_1 & y_1 & x_1^2 & x_1y_1 & y_1^2 \\
                                1 & x_2 & y_2 & x_2^2 & x_2y_2 & y_2^2 \\
                                1 & x_3 & y_3 & x_3^2 & x_3y_3 & y_3^2 \\
                                1 & x_4 & y_4 & x_4^2 & x_4y_4 & y_4^2 \\
                                1 & x_5 & y_5 & x_5^2 & x_5y_5 & y_5^2 \\
                                1 & x_6 & y_6 & x_6^2 & x_6y_6 & y_6^2 
                        \end{array}\right] = \begin{bmatrix}
                            V_0 & | & V_1 & | & V_2
                        \end{bmatrix}.
        \end{align*}
        We get that
        \begin{align*}
            t_0&=\rank(V_{\leq 0})=1,\\
             t_1&=\rank(V_{\leq 1})-\rank(V_{\leq 0})=3-1=2,\\ t_2&=\rank(V_{\leq 2})-\rank(V_{\leq 1})=6-3=3.
        \end{align*}
        The results in Figure~\ref{fig:non-alg-var} confirm the scaling of the eigenvalues also in this case.
        % \begin{figure}
        %     \centering
        %     \includegraphics[width=0.9\textwidth]
        % {nonsingularVan.png}
        %     \caption{General position geometry. Eigenvalues of Dirichlet kernel with samples on the grid.}
        %     \label{fig:non-alg-var}
        % \end{figure}
    \end{enumerate}

% WHAT IS THIS FIGURE?
% \begin{figure}[h]
%     \centering
%     \includegraphics[width=0.6\textwidth]{parabola.png}
%     \caption{The 4 nodes are of the form $(x_i,x_i^2)$. slope of $K_{1,x}$ is $3.87 \approx 4 = 2m$. $rank(V_{\leq m})=4 \rightarrow m=2$
%     \todo[inline]{What is this figure?}
%     }
%     \label{fig:parabola}
% \end{figure}

\subsection{Super-resolution}
In this section, we consider the super-resolution problem for multidimensional sparse measures described in the Introduction, as it was the original motivation for our investigations. We reconstruct measures of the form \eqref{eq:spike} from the noisy measurements \eqref{eq:samples}. We consider the asymptotic regime of constant $N$ and $\Delta\to 0$.
\subsubsection{Local Stability}
    Analogous to some previous studies in the one-dimensional case, e.g. \cite{batenkov2013c,batenkov2018-genspikes}, we consider the ``local stability'' of the problem to be well-approximated by the error incurred by the nonlinear least squares (NLS) method applied to the noisy measurements:
    $$
    \left\{\hat{x}_j^{NLS}, \hat{\alpha}_j^{NLS}\right\}_{j=1}^n = \arg\min_{\hat{\alpha},\hat{x}} \frac{1}{2}\sum_{k\in\Gg_N} \left| \hat{f}(k)-\sum_{j=1}^n \hat{\alpha}_j e^{\imath \left\langle k, \hat{x}_j \right\rangle} \right|^2.
    $$
    The initial values for the NLS method are taken to be the true parameter values (which are clearly unknown in practice). Our implementation of the NLS method is based on the Levenberg-Marquardt algorithm, with the complex residuals converted to real residuals by concatenating the real and imaginary parts. The results are shown in Figures \ref{fig:nls-d2} and \ref{fig:nls-d3}. To measure the accuracy, we fitted the scaled error
    $$
    \kappa:=\max_{j=1,\dots,n} \| \hat{x}_j^{NLS} - x_j \|_{\infty} / \sigma
    $$
    to the curve $\kappa \sim \Delta^{m}$, where $\Delta$ is the minimal distance between the nodes. It can be seen that the accuracy crucially depends on the geometry, and is indeed best for the random configuration, and worst for the single-line configuration (for which the scaling is precisely $\kappa \sim \Delta^{2n-2}$ for any $d$, as established in previous works on the subject).
    \begin{figure}[h]
        \centering
        \includegraphics[width=0.3\textwidth]{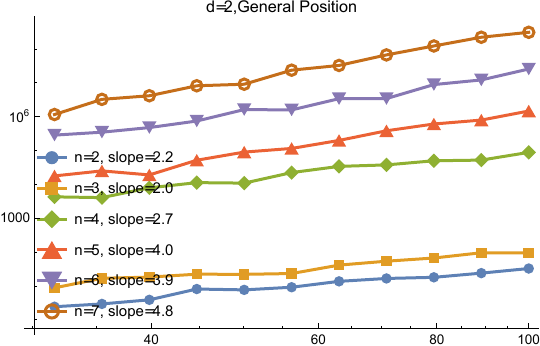}
        \includegraphics[width=0.3\textwidth]{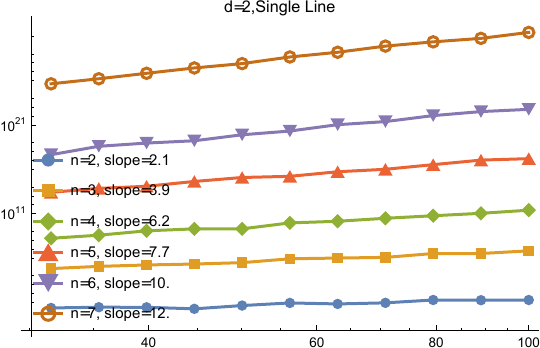}
        \includegraphics[width=0.3\textwidth]{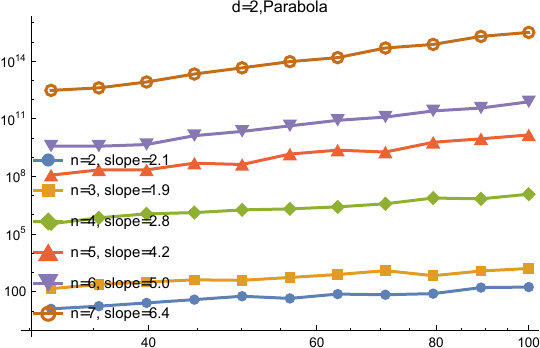}
        \caption{NLS reconstruction, general position/line/parabola, $d=2$.}
        \label{fig:nls-d2}
    \end{figure}

    \begin{figure}[h]
        \centering
        \includegraphics[width=0.3\textwidth]{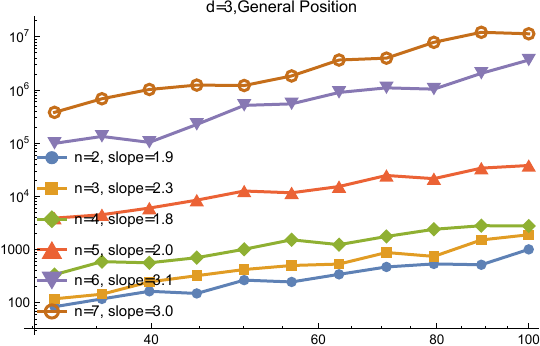}
        \includegraphics[width=0.3\textwidth]{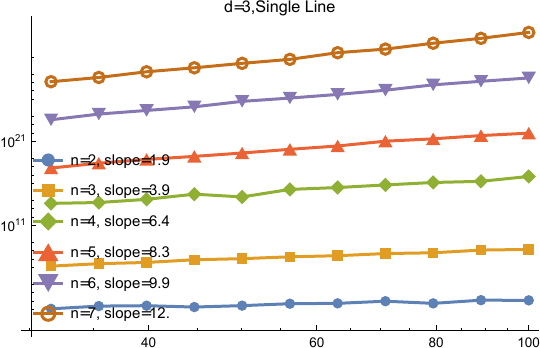}
        \includegraphics[width=0.3\textwidth]{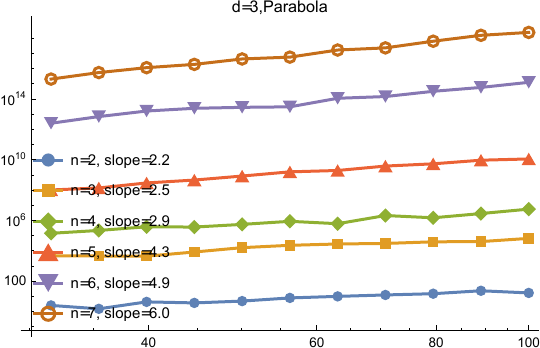}
        \caption{NLS reconstruction, general position/line/parabola, $d=3$.}
        \label{fig:nls-d3}
    \end{figure} 

\subsubsection{2D-ESPRIT}

We have implemented the 2D ESPRIT algorithm according to \cite[Section IV-A]{sahnoun2017}. The experimental setup was as follows:
\begin{enumerate}
    \item The number of nodes varied from $n=2$ to $n=6$.
    \item The geometry was either random, a line $y=5x$, or a parabola $y=10x^2$.
    \item The ESPRIT algorithm parameters were set as follows: $M_1=M_2=40,\; L_1=L_2=10,\;\beta_1=\beta_2=0.5$.
    \item The noise level was set to $\sigma=10^{-20}$.
\end{enumerate} 
The results are shown in Figure \ref{fig:esprit-all}. As for the NLS method, we fitted the scaled error
$$
\kappa:=\max_{j=1,\dots,n} \| \hat{x}_j^{ESPRIT} - x_j \|_{\infty} / \sigma
$$
to the curve $\kappa \sim \Delta^{m}$, where $\Delta$ is the minimal distance between the nodes. As in the NLS case, it can be seen that the accuracy crucially depends on the geometry, providing best scaling for the random configuration, and worst scaling for the single-line configuration.

\begin{figure}[h]
    \centering
    \includegraphics[width=0.3\textwidth]{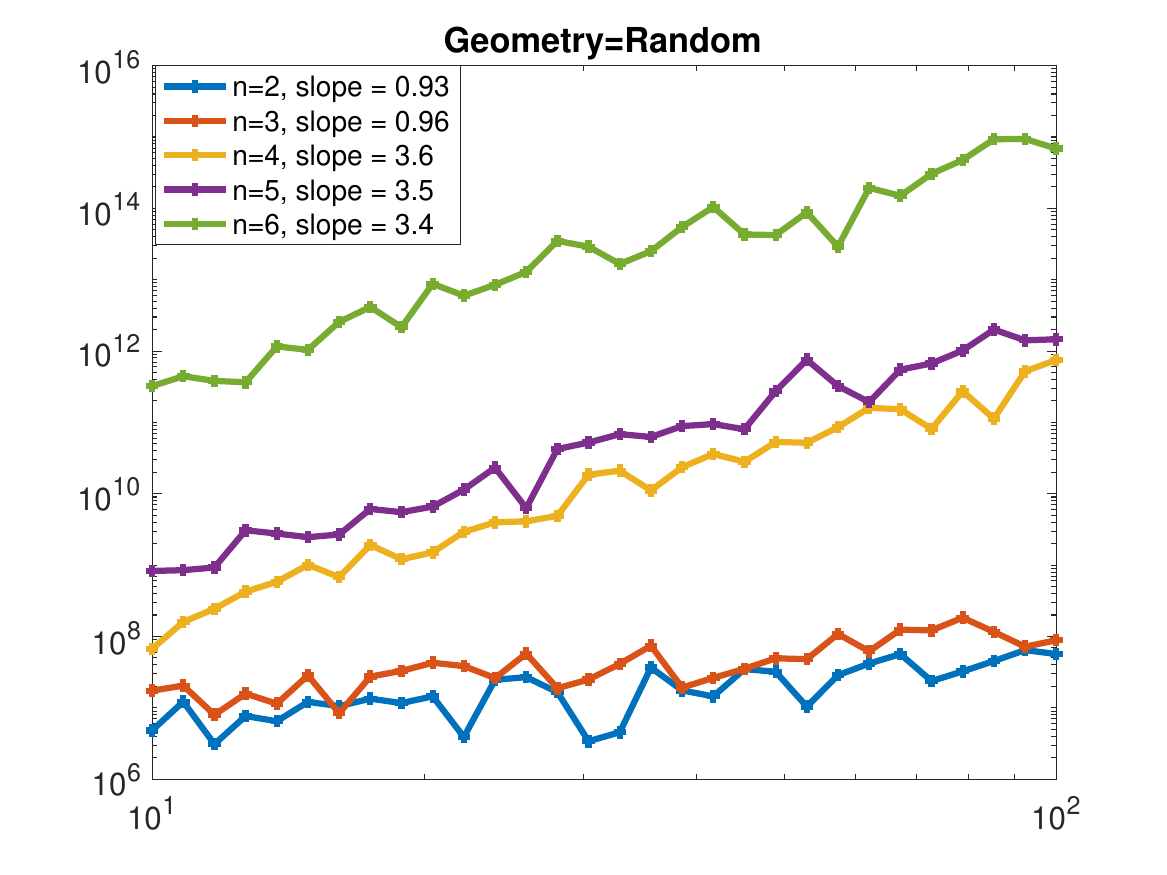}
    \includegraphics[width=0.3\textwidth]{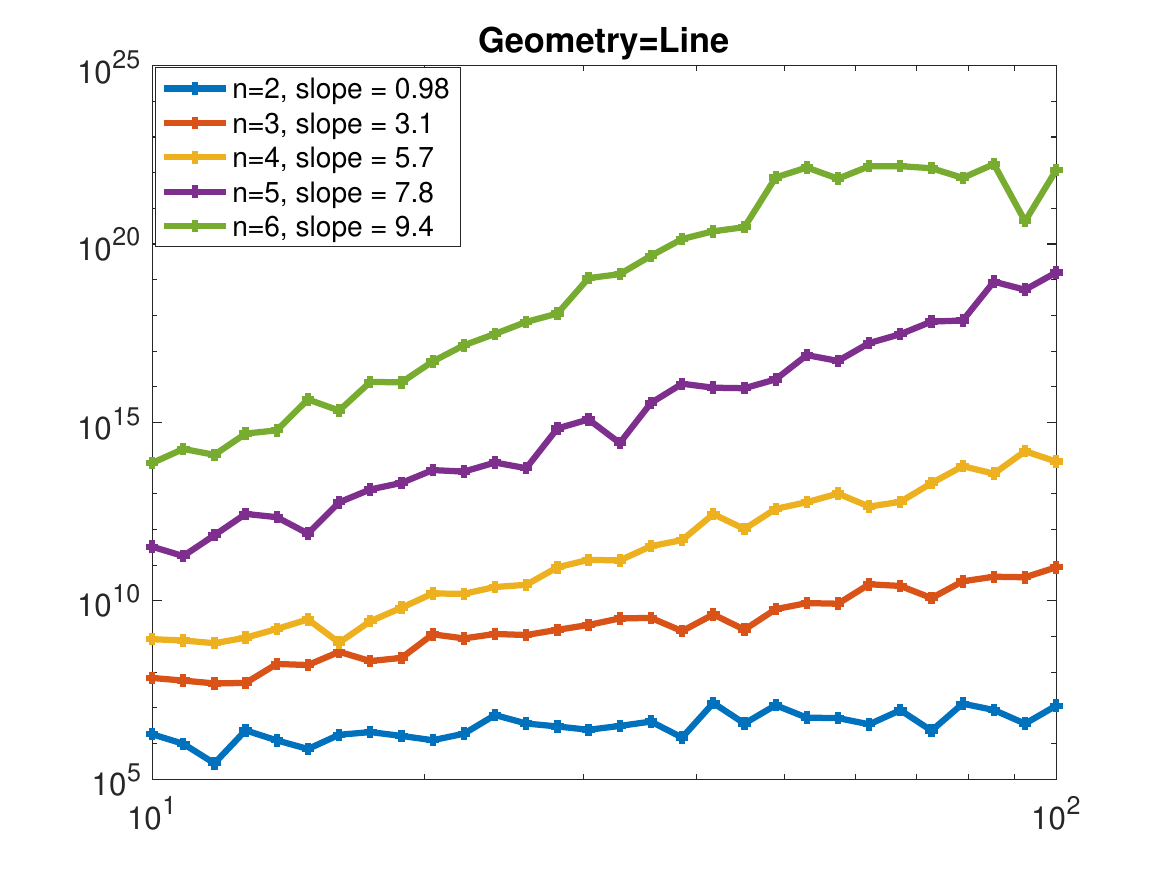}
    \includegraphics[width=0.3\textwidth]{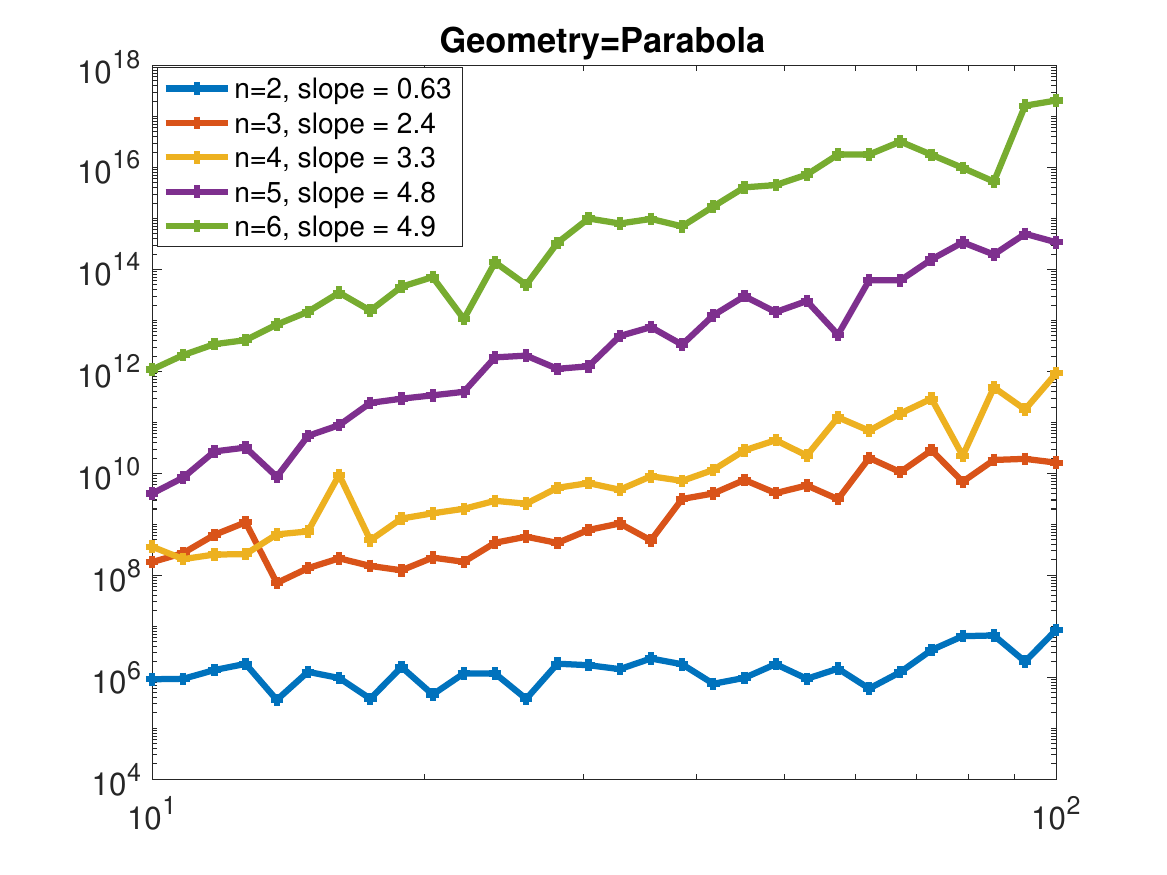}
    \caption{Accuracy of 2D ESPRIT for different geometries.}
    \label{fig:esprit-all}
\end{figure}

\appendix
\section{Taylor Expansion}\label{app:taylor}
 Let $\mathcal{K} \in \mathcal{C}^{(r,r)}(\Omega)$ where $\Omega$ is an open neighborhood of $0$ and $r \geq m+1$, where $m=\mu(\ci)$ is the discrete moment order of $\ci$ (Definition \ref{def:mu-def}). We recall the Maclaurin expansion for bivariate functions from \cite[section 5]{usevich2021}. For $\epsilon x,\epsilon y$ such that $[0,\epsilon x], [0,\epsilon y] \subset \Omega$ (where $[0,x]$ is a line segment from $0$ to $x$), we have:
\begin{align*}
    \mathcal{K}_{\epsilon}(x,y) = \mathcal{K}(\epsilon x,\epsilon y) &= \sum_{\alpha,\beta \in \PP_m}\frac{{(\epsilon x)}^{\alpha}{(\epsilon y)}^{\beta}}{\alpha ! \beta !}\mathcal{K}^{(\alpha,\beta)}(0,0) \\
    &+ \sum_{\alpha \in \PP_m, \beta \in \HH_{m+1}}\frac{{(\epsilon x)}^{\alpha}{(\epsilon y)}^{\beta}}{\alpha ! \beta !}\mathcal{K}^{(\alpha,\beta)}(0,\theta_{\epsilon y, \alpha}\epsilon y) 
    + \sum_{\alpha \in \HH_{m+1}, \beta \in \PP_m}\frac{{(\epsilon x)}^{\alpha}{(\epsilon y)}^{\beta}}{\alpha ! \beta !}\mathcal{K}^{(\alpha,\beta)}(\eta_{\epsilon x, \beta}\epsilon x,0) \\
    &+ \sum_{\alpha,\beta \in \HH_{m+1}}\frac{{(\epsilon x)}^{\alpha}{(\epsilon y)}^{\beta}}{\alpha ! \beta !}\mathcal{K}^{(\alpha,\beta)}(\psi_{\epsilon x, \epsilon y}\epsilon x,\xi_{\epsilon x, \epsilon y}\epsilon y),
\end{align*}

where $\{\theta_{\epsilon y, \alpha}\}_{\alpha \in \PP_m} \subset [0,1]$ depend on $\epsilon y$, $\{\eta_{\epsilon x, \beta}\}_{\beta \in \PP_m} \subset [0,1]$ depend on $\epsilon x$ and $\psi_{\epsilon x,\epsilon  y},\xi_{\epsilon x,\epsilon  y}$ depend on both $\epsilon x$ and $\epsilon y$.
\par Let $\PP_m = \{\alpha_1,\dots, \alpha_{p_m}\}$. Then we can write the above expansion as follows:
\begin{align*}
    &\mathcal{K}(\epsilon x,\epsilon y) = \big[{(\epsilon x)}^{\alpha_1},\dots,{(\epsilon x)}^{\alpha_{p_m}}\big]W_{\leq m}{\big[{(\epsilon y)}^{\alpha_1},\dots,{(\epsilon y)}^{\alpha_{p_m}}\big]}^T \\
    &+ \epsilon^{m+1}\big[{(\epsilon x)}^{\alpha_1},\dots,{(\epsilon x)}^{\alpha_{p_m}}\big]\mathbf{w}_{1,y}(\epsilon) + \epsilon^{m+1}{\mathbf{w}_{2,x}(\epsilon)}^T{\big[{(\epsilon y)}^{\alpha_1},\dots,{(\epsilon y)}^{\alpha_{p_m}}\big]}^T \\
    &+ \epsilon^{2(m+1)}w_{3,x,y}, \\
    &\mathbf{w}_{1,y}(\epsilon) = {\begin{bmatrix}
        \sum_{\beta \in \HH_{m+1}}\frac{\mathcal{K}^{(\alpha_1,\beta)}(0,\theta_{\epsilon y, \alpha_1}\epsilon y)}{\alpha_1! \beta!}y^{\beta}, \dots, \sum_{\beta \in \HH_{m+1}}\frac{\mathcal{K}^{(\alpha_{p_m},\beta)}(0,\theta_{\epsilon y, \alpha_{p_m}}\epsilon y)}{\alpha_{p_m}! \beta!}y^{\beta} 
    \end{bmatrix}}^T \\
    %= \begin{bmatrix}
     %   \sum_{\beta \in \HH_{m+1}}\frac{K^{(\alpha_k,\beta)}(0,\theta_{\epsilon y, \alpha_k}\epsilon y)}{\alpha_k! \beta!}{(\epsilon y)}^{\beta}
    %\end{bmatrix}_{k=1,\dots,p_m} \\
    &\mathbf{w}_{2,x}(\epsilon) = \begin{bmatrix}
        \sum_{\alpha \in \HH_{m+1}}\frac{\mathcal{K}^{(\alpha,\beta)}(\eta_{\epsilon x,\beta}(\epsilon x),0)}{\alpha! \beta!}x^{\alpha}
    \end{bmatrix}_{\beta \in \PP_m}^T\\
    &w_{3,x,y}(\epsilon) = \sum_{\alpha, \beta \in \HH_{m+1}}\frac{\mathcal{K}^{(\alpha,\beta)}(\psi_{\epsilon x, \epsilon y}(\epsilon x),\xi_{\epsilon x, \epsilon y}(\epsilon y))}{\alpha! \beta!}{(\epsilon x)}^{\alpha}{(\epsilon y)}^{\beta}
\end{align*}
where $\mathbf{w}_{2,x}, \mathbf{w}_{1,y}: [0,\epsilon_0] \rightarrow \RR^{p_m}$ are bounded vector functions with $\epsilon_0 > 0$ such that $\epsilon_0 x, \epsilon_0 y \in \Omega$, $\epsilon \in [0,\epsilon_0]$ and $w_{3,x,y}$ is bounded.
\par Let $\epsilon_0 > 0$ such that $\epsilon_0 x_i \in \Omega$ for $x_i \in \ci$. Thus for $0 \leq \epsilon \leq \epsilon_0$, the scaled kernel matrix $K_{\epsilon}$ admits the following expansion:
\begin{equation}
        K_\epsilon = V_{\leq m}\Delta_m W \Delta_m V_{\leq m}^T + \epsilon^{m+1}(V_{\leq m}\Delta_m W_1(\epsilon)+W_2(\epsilon)\Delta_m V_{\leq m}^T)+\epsilon^{2{ (m+1)}}W_3(\epsilon) 
    \end{equation}
where $W_1(\epsilon) := \big[\mathbf{w}_{1,x_1}(\epsilon), \dots, \mathbf{w}_{1,x_n}(\epsilon)\big]$, $W_2(\epsilon) :={\big[\mathbf{w}_{2,x_1}(\epsilon), \dots, \mathbf{w}_{2,x_n}(\epsilon)\big]}^T$ and $W_3(\epsilon) = \big[w_{3,x_i,x_j}(\epsilon)\big]_{i,j=1}^{n,n}$.
\section{Auxiliary Lemmas}

\begin{lemma}\label{app:lem}
  Let $D \in \RR^{k\times k}$ be a $k \times k$ diagonal matrix and $P$ a permutation matrix corresponding to a permutation $\sigma$, then
  \[ PDP^T = \diag\{d_{\sigma(1)},\dots,d_{\sigma(k)}\}, \quad D=\diag\{d_1,\dots,d_k\}.\]
\end{lemma}

\begin{proof}
    If $P$ is a permutation matrix corresponding to a permutation $\sigma$, then for every vector $v\in\RR^k$
$$
Pv = \begin{bmatrix}
    v_{\sigma(1)} \\
    \vdots \\
    v_{\sigma(k)}
\end{bmatrix}
$$
and then for every matrix $A=\begin{bmatrix}R_1\\\vdots\\R_k\end{bmatrix}\in\RR^{k\times k}$, we have
$$
PA = \begin{bmatrix}
    R_{\sigma(1)} \\
    \vdots \\
    R_{\sigma(k)}
\end{bmatrix}.
$$
Now let $A=D=\diag\{d_1,\dots,d_k\}$, and take $j=1,\dots,k$. Then $R_{\sigma(j)}=d_{\sigma(j)} e_{\sigma(j)}^T$. Therefore
$$
R_{\sigma(j)} P^T = d_{\sigma(j)}\left(P e_{\sigma(j)}\right)^T = d_{\sigma(j)} e_j^T.
$$
Stacking all the rows we get $PDP^T=\diag\{d_{\sigma(1)},\dots,d_{\sigma(k)}\}$.
\end{proof}
%\begin{proof}
 %   Since $D$ is a real diagonal matrix, $d_1,...,d_k$ are it's eigenvalues with $e_{j_1},...,e_{j_k}$ corresponding standard eigenvectors. Thus for $d_{\ell_1} \geq d_{\ell_2} \geq ... d_{\ell_k}$ we have
%    \[
 %       PDP^T = \Lambda, \quad P:=\big(e_{j_1}|...|e_{j_k}\big), \Lambda:=diag( d_{\ell_2},..., d_{\ell_k}).
 %   \]
%   The same applies for $\tilde{D}$, $\tilde{P}\tilde{D}\tilde{P}^T = \Lambda$. Thus we have
%    \[
%        \tilde{D} = \hat{P}D\hat{P}^T, \quad \hat{P}:=P^T\tilde{P}.
%    \]
%\end{proof}

\begin{lemma}\label{app:det-lem}
    Let $M, B \in \RR^{n \times n}$. For $\epsilon > 0$ we have
    \begin{equation}
        \det(M + \epsilon B) = \det(M) + O(\epsilon),\quad \epsilon \ll 1.
    \end{equation}
\end{lemma}
\begin{proof}
    Recall formula (0.8.12.3) for sum of matrices in \cite{horn2013}:
    \begin{equation}
        \det(M + \epsilon B) = \sum_{k=0}^n \epsilon^{k}\operatorname{tr}(\operatorname{adj}_k(M)C_k(B)),
    \end{equation}
    where $\operatorname{tr}(\cdot)$ is the trace of matrix, $\operatorname{adj}_k(\cdot)$ is the $k$-th adjugate of matrix (0.8.12) and $C_k(\cdot)$ is the $k$-th compound (0.8.1). Using the following properties: $\operatorname{adj}_0(M) = \det(M)$ and $C_0(M)=1$, we have
    \begin{equation}
        \det(M + \epsilon B) = \epsilon^0 \operatorname{tr}(\operatorname{adj}_0(M)C_0(B)) + O(\epsilon) = \det(M) + O(\epsilon),
    \end{equation}
    completing the proof.
\end{proof}
%\bibliographystyle{plain}
% \bibliography{multi-sr}

\printbibliography
\end{document}